\documentclass[12pt]{amsart}
\usepackage{amsfonts}
\usepackage{mathrsfs}
\usepackage{amssymb}
\usepackage{amsmath}
\usepackage{graphicx}
\usepackage{tikz}
\usepackage{tikz-cd}
\usepackage{caption}
\usetikzlibrary{arrows}
\usepackage{color}
\usepackage{mathtools}
\usepackage[utf8]{inputenc}
\usepackage{hyperref}
	
\newcommand{\masterseries}{\Omega}
\newcommand{\simplifiedmasterseries}{\Omega_{\mathrm{ev}}}
\newcommand{\simplifiedmasterseriestwo}{\Omega'_{\mathrm{ev}}}

\newcommand{\la}{\lambda}

\newcommand{\partialpairing}{\!-}

\newcommand{\gluing}{\mathrm{glue}}
\newcommand{\capping}{\mathrm{cap}}

\newcommand{\fsl}{\mathfrak sl}

\newcommand{\CC}{\mathcal{C}}
\newcommand{\CM}{\mathcal{M}}
\newcommand{\ins}{\mathrm{ins}}
\newcommand{\M}{\overline{\CM}}

\newcommand{\DD}{\mathcal{D}}

\newcommand{\OO}{\mathcal{O}}
\newcommand{\CF}{\mathcal{F}}
\newcommand{\CH}{\mathrm{CH}}
\newcommand{\CL}{\mathcal{L}}

\newcommand{\J}{\mathcal{J}}

\newcommand{\Poli}{\mathsf{Poli}}
\newcommand{\QQ}{\mathbb{Q}}

\newcommand{\too}{\rightarrow}

\newcommand{\Pic}{\mathrm{Pic}}
\newcommand{\D}{\qquad}
\newcommand{\N}{\mathbb{N}}
\newcommand{\id}{\mathrm{id}}

\newcommand{\rescaledpsi}{\phi}
\DeclareMathOperator{\pullbackmap}{pb}
\DeclareMathOperator{\ev}{ev}
\newcommand{\rescaledkappa}{\xi}

\newcommand{\sle}{\mathfrak{e}}
\newcommand{\slf}{\mathfrak{f}}
\newcommand{\slh}{\mathfrak{h}}

\newcommand{\lf}{\ell\!f}

\newtheorem{thm}[subsection]{Theorem}
\newtheorem{cor}[subsection]{Corollary}
\newtheorem{lem}[subsection]{Lemma}
\newtheorem{prop}[subsection]{Proposition}
\newtheorem{prop-def}[subsection]{Proposition-Definition}
\newtheorem{dfn}[subsection]{Definition}
\theoremstyle{remark}
\newtheorem{rem}[subsection]{Remark}
\newtheorem{notation}[subsection]{Notation}

\newtheorem{ex}[subsection]{Example}
\newtheorem{conj}[subsection]{Conjecture}

\newtheorem{con}[subsection]{Conventions}

\makeatletter
\let\c@equation=\c@subsection

\makeatother

\begin{document}

\title[The Polishchuk differential operator via surfaces]
{An action of the Polishchuk differential operator via punctured surfaces}
\author[G.~C.~Drummond-Cole and M.~Tavakol]{Gabriel C. Drummond-Cole and Mehdi Tavakol}
   \address{Center for Geometry and Physics, Institute for Basic
  Science,  Pohang, Republic of Korea 37673 }
   \email{gabriel@ibs.re.kr}
 \address{School of mathematics and statistics, University of Melbourne, VIC 3010, Australia}
\email{mehdi.tavakol@unimelb.edu.au}

\begin{abstract}
For a family of Jacobians of smooth pointed curves there is a notion of tautological algebra.
There is an action of $\fsl_2$ on this algebra.
We define and study a lifting of the Polishchuk operator, corresponding to $\slf \in \fsl_2$, on an algebra consisting of punctured Riemann surfaces.
As an application we prove that a collection of tautological relations on moduli of curves, discovered by Faber and Zagier, come from a class of relations on the universal Jacobian. 
\end{abstract}   

\maketitle

\section{Introduction}

The study of algebraic cycles on moduli spaces of curves was initiated by Mumford in the influential article \cite{M}.
He developed intersection theory on such moduli spaces and defined the notion of \emph{tautological classes}.
These are the most natural algebraic cycles on the moduli space and many geometric constructions lead to tautological classes.
The collection of tautological cycles generate a distinguished subring of the Chow ring, known as the \emph{tautological ring}.
A fundamental open question concerning tautological rings is to understand the space of all relations among tautological classes.
The purpose of this note is to study the connection between two classes of tautological relations.
The first class of relations was discovered by Faber and Zagier around 2000 in an unpublished work.
The second class is based on the study of tautological classes on the universal Jacobian by Yin \cite{Y1}.
The method of Yin gives a powerful tool to produce a large class of tautological relations on moduli of curves.
Conjecturally, his method should give a complete description of tautological rings.
The main ingredient in his approach is the Polischuck differential operator $\DD$ which acts on the tautological ring of the universal Jacobian.
We will show that there is a natural lifting of this differential operator to an algebra built from punctured Riemann surfaces.
The lifting of the operator $\DD$ corresponds to gluing Riemann surfaces along the punctures and closing punctures with open discs.
Using this combinatorial interpretation of $\DD$ we are able to find closed formulas for a certain class of tautological relations from the Jacobian side.
We analyze these relations and show that they match with a class of Faber--Zagier relations.

\begin{con}
Throughout this note we consider algebraic cycles modulo rational equivalence.
Chow rings and cohomology rings are taken with $\QQ$-coefficients. 
\end{con}

\vspace{+10pt}
\noindent{\bf Acknowledgments.}
The first author was supported by the Institute for Basic Science under IBS-R003-D1.
The second author was supported by the Institute for Basic Science under IBS-R003-S1, by the Max Planck institute for mathematics, and by the Australian Research Council grant DP180103891.

\section{Tautological classes on the universal Jacobian}
The tautological ring of a fixed Jacobian variety under algebraic equivalence was defined and studied by Beauville \cite{B}. 
Tautological rings of families of Jacobian varieties under rational equivalence have been studied extensively since then.
For more details see the references \cite{MO, MP, PO1, PO2, PO3}.
Here we consider the relative version of this story under rational equivalence following Yin \cite{Y1} to which we refer the reader for precise definitions.
Let $\pi: \CC \too S$ be a family of smooth curves of genus $g\geq 2$ which admits a section $s: S \too \CC$. 
In this article we will assume that the base scheme $S$ is the universal curve $\CC_g=\CM_{g,1}$.
Consider the relative Picard group $\J_g=\Pic^0(\CC/S)$ of divisors of degree zero.
It is an abelian scheme of relative dimension $g$ over the base $S$.
The section $s$ induces an injection $\iota: \CC \too \J_g$ from $\CC$ into the universal Jacobian $\J_g$.
The geometric point $x$ on a curve $C$ is sent to the line bundle $\OO_C(x-s)$ via the morphism $\iota$. 
For an integer $k$ consider the associated endomorphism of $\J_g$ induced by multiplication with $k$ on fibers of the family $\J_g \to S$.

\begin{dfn}
For integers $i,j$ the subgroup $\CH^i_{(j)}(\J_g)$ of the Chow group $\CH^i(\J_g)$ is defined as all degree $i$ classes on which the morphism $k^*$ acts via multiplication with $k^{2i-j}$.
Equivalently, the action of the morphism $k_*$ on $\CH^i_{(j)}(\J_g)$ is multiplication by $k^{2g-2i+j}$. 
\end{dfn}

\begin{prop}
The Beauville decomposition of the Chow group of $\J_g$ has the form
$\CH^*(\J_g)=\bigoplus_{i,j} \CH_{(i,j)}(\J_g)$, where $\CH_{(i,j)}(\J_g) \coloneqq \CH_{(j)}^{\frac{i+j}{2}}(\J_g)$ for $i \equiv j$ mod 2.
\end{prop}
In the following the Chow group of the universal Jacobian equipped with the intersection product is denoted by $(\CH^*(\J_g), .)$ and we usually drop the product sign to simplify the notation.
But there is another product on the Chow group of $\J_g$:

\begin{dfn}
The \emph{Pontryagin product} $x * y$ of two algebraic cycles $x,y \in \CH^*(\J_g)$ is defined as $\mu_*(\pi_1^* x \cdot \pi_2^* y)$, where $\pi_1,\pi_2: \J_g \times_S \J_g \too \J_g$ are the natural projections and $\mu: \J_g \times_S \J_g \too \J_g$ is the addition morphism.
\end{dfn}
Recall that to an abelian scheme $A/S$ there is an associated Poincar\'e line bundle $\mathcal{P}$ on $A \times_S A^t$ trivialized along the zero sections.
Here, $A^t$ denotes the dual abelian scheme $\Pic^0_{A/S}$.

\begin{dfn}
A \emph{polarization} of $A/S$ is a symmetric isogeny $\la: A \to A^t$ such that the pullback of the Poincar\'e bundle via the morphism $(\id_A, \la): A \to A \times_S A^t$ is relatively ample over $S$. We say that $\la$ is a \emph{principal polarization} when $\la$ defines an isomorphism. 
\end{dfn}
Any such polarization induces a line bundle $\CL_{\la}$ in the rational Picard group of $A$ with the following properties:
\begin{itemize}
\item
The line bundle $\CL_{\la}$ is relatively ample over $S$,
\item
It is symmetric, i.e. $[-1]^* \CL_{\la}=\CL_{\la}$,
\item
It is trivialized along the zero section.
\end{itemize}
The first Chern class of $\CL_{\la}$ is called \emph{the universal theta divisor} and it will be denoted by $\theta$.
For more details we refer the reader to \cite[Chapter 2]{Y1}.
Let $\ell$ be the first Chern class of the Poincar\'e bundle.

\begin{dfn}
The Fourier--Mukai transform $\CF$ is defined as \[
\CF(x)=\pi_{2,*}(\pi_1^* x \cdot \exp(\ell)).
\]
It gives an isomorphism between $(\CH^* (\J_g),.)$ and $(\CH^*(\J_g),*)$.
 
\end{dfn}
We now recall the definition of the tautological ring of $\J_g$:

\begin{dfn}
The tautological ring $R^*(\J_g)$ of $\J_g$ is defined as the smallest $\QQ$-subalgebra of the rational Chow ring $\CH^*(\J_g)$ which contains the class of $\iota_*[\CC]$ and is stable under the Fourier--Mukai transform and all maps $k^*$ for integers $k$. 
\end{dfn}

\begin{rem}
It follows that for an integer $k$ the tautological algebra becomes stable under $k_*$ as well.
\end{rem}

The generators of $R^*(\J_g)$ are expressed in terms of the components of the curve class in the Beauville decomposition.
Define the following classes:
\[
p_{i,j} \coloneqq \CF \left(\theta^{\frac{j-i+2}{2}} \cdot \iota_* [\CC]_{(j)} \right) \in \CH_{(i,j)}(\J_g).
\]
We have that $p_{2,0}=-\theta$ and $p_{0,0}=g[\J_g]$. 
The class $p_{i,j}$ vanishes for $i<0$ or $j<0$ or $j>2g-2$.
The tautological class $\psi$ is defined as $\psi \coloneqq s^*(K)$, where $K$ is the first Chern class of the relative dualizing sheaf of the morphism $\CC \to S$.
The pullback of $\psi$ via the natural map $\J_g \too S$ is denoted by the same letter. 
The following fact is proved in \cite[Theorem 3.6]{Y1}:

\begin{thm}
\label{thm: taut Jg generation}
The tautological ring of $\J_g$ is generated by the classes $\{p_{i,j}\}$ and $\psi$.
In particular, it is finitely generated. 
\end{thm}

\subsection{Lefschetz decomposition of Chow groups}
\label{subsec: Lefschetz decomp}
The action of the Lie algebra $\fsl_2$ on the Chow groups of a fixed abelian variety was studied by K{\"u}nnemann~\cite{KU}.
Polishchuk \cite{PO1} studied the $\fsl_2$ action for abelian schemes which works over families. 
We follow the standard convention that $\fsl_2$ is generated by elements $\sle,\slf,\slh$ satisfying: 
\[
[\sle,\slf]=\slh, \D [\slh,\sle]=2\sle, \D [\slh,\slf]=-2\slf.
\]
With this notation the action of $\fsl_2$ on Chow groups of $\J_g$ is given by
\[\sle: \CH_{(j)}^i(\J_g) \too \CH_{(j)}^{i+1}(\J_g) \D x \too -\theta \cdot x,\]
\[\slf: \CH_{(j)}^i(\J_g) \too \CH_{(j)}^{i-1}(\J_g) \D x \too -\frac{\theta^{g-1}}{(g-1)!} * x,\]
\[\slh: \CH_{(j)}^i(\J_g) \too \CH_{(j)}^i(\J_g) \D x \too -(2i-j-g) x,\]
The operator $\slf$ restricted to the tautological ring of $\J_g$ is given by the following differential operator:
\begin{align*}
\nonumber\mathcal{D}&=\frac{1}{2} \sum_{i,j,k,l} \left( \psi p_{i-1,j-1}p_{k-1,l-1}- \binom{i+k-2}{i-1}p_{i+k-2,j+l} \right) \partial_{p_{i,j}} \partial_{p_{k,l}}
\\&\quad{}+\sum_{i,j} p_{i-2,j}\partial_{p_{i,j}}.
\end{align*}
\section{Faber--Zagier relations}
\label{section: FZ relations}

Let $g \geq 2$ and consider the moduli space $\CM_g$ of smooth curves of genus $g$. 
Consider the universal curve $\pi: \CC_g \too \CM_g$ and denote by $\omega_{\pi}$ its relative dualizing sheaf. 
The first Chern class of $\omega_\pi$ is denoted by $K$.
In \cite{M} Mumford defined \emph{the kappa class} $\kappa_i$ as the push-forward $\pi_*(K^{i+1})$.
It is an algebraic cycle of degree $i$. Notice that $\kappa_0=2g-2$.

\begin{dfn}
The tautological ring $R^*(\CM_g)$ of $\CM_g$ is defined as the $\QQ$-subalgebra of the rational Chow ring $\CH^*(\CM_g)$ of $\CM_g$ generated by kappa classes.
\end{dfn}
In unpublished work Faber and Zagier studied the Gorenstein quotient of the tautological ring of $\CM_g$.
Recall that there is an isomorphism \[
\Phi: R^{g-2}(\CM_g) \cong \QQ.
\]
This follows from a result of Looijenga \cite{L} which states that $R^{g-2}(\CM_g)$ is at most one-dimensional and from the result of Faber \cite{F2} which shows that $\kappa_{g-2}$ is nonzero.
There is a natural way to extend $\Phi$ to a group homomorphism
 \[
\Phi: R^*(\CM_g) \too \QQ,
\]
by requiring that any element is sent to zero unless it is of degree $g-2$.
Each element $x$ of the tautological ring $R^*(\CM_g)$ defines a linear map \[
\Phi_x: R^*(\CM_g) \too \QQ
\] that sends an element $y \in R^*(\CM_g)$ to $\Phi(x \cdot y) \in \QQ$.

\begin{dfn}
The Gorenstein quotient of the ring $R^*(\CM_g)$, denoted $G^*(\CM_g)$, is the quotient of $R^*(\CM_g)$ by the ideal generated by all elements $x$ for which $\Phi_x$ defines the zero map.
\end{dfn}
Let \[
\mathbf{p}=\{p_1, p_3,p_4,p_6,p_7,p_9,p_{10}, \dots \}
\]
be a variable set indexed by the positive integers not congruent to 2 mod 3. 
The formal power series $\Psi$ is defined by the formula:
\[
\Psi(t, \mathbf{p})=\bigg(1+t p_3+t^2 p_6 + t^3 p_9 + \dots \bigg) \cdot \sum_{n=0} ^{\infty} \frac{(6n)!}{(3n)!(2n)!} t^n
\]
\[
\qquad \qquad + \bigg(p_1+t p_4 + t^2 p_7 + \dots \bigg) \cdot \sum_{n=0} ^{\infty} \frac{(6n)!}{(3n)!(2n)!} \frac{6n+1}{6n-1} t^n.
\]
Let $\sigma$ be a partition of $|\sigma|$ with parts not congruent to 2 modulo 3. 
For such partitions define rational numbers $\alpha_n(\sigma)$ as follows: 
\[
\log(\Psi(t,\mathbf{p}))= \sum_{\sigma} \sum_{n=0}^{\infty} \alpha_n(\sigma) t^n \mathbf{p}^{\sigma},
\]
where for $\sigma$ the partition 
$[1^{a_1} 3^{a_3} 4^{a_4} \dots]$, we use $\mathbf{p}^{\sigma}$ to denote the monomial 
$(p_1^{a_1} p_3^{a_3} p_4^{a_4} \dots)$. 
Define \[
\gamma \coloneqq \sum_{\sigma} \sum_{n=0} ^{\infty} \alpha_n(\sigma) \kappa_n t^n \mathbf{p}^{\sigma};
\] 
then the relation 

\begin{equation*}
\label{fz}
[\exp(-\gamma)]_{t^n \mathbf{p}^{\sigma}}=0
\end{equation*}
holds in the Gorenstein quotient $G^*(\CM_g)$ of $R^*(\CM_g)$ when 
$g-1+|\sigma| <3n $ and $g \equiv n+|\sigma|+1$ (mod 2). 
In 2013 Pandharipande and Pixton \cite{PP} proved that Faber--Zagier relations hold in the tautological ring of $\CM_g$.
It is an open question whether all relations in the tautological ring follow from Faber--Zagier relations.

\subsection{Relations on moduli of curves from the universal Jacobian}
The differential operator $\mathcal{D}$ provides a powerful tool to produce tautological relations. 
The crucial property of $\mathcal{D}$ is that it preserves the rational equivalence of algebraic cycles.
Therefore, if we start from a collection of tautological relations we can produce a larger class by applying the differential operator $\mathcal{D}$.
These relations can be used to produce tautological relations on moduli of curves.
Assume that the base scheme $S$ is the universal curve $\CC_g=\CM_{g,1}$ as before.

\begin{dfn}
The tautological ring $R^*(S)$ of $S$ is defined to be the $\QQ$-subalgebra of the Chow ring $\CH^*(S)$ generated by kappa classes and the class of the relative dualizing sheaf $\omega_\pi$ of $\pi: S \to \CM_g$.
\end{dfn}

\begin{rem}\label{jacobian_curve}
Consider the natural morphism $\pi: \J_g \to S$.
According to \cite[Corollary 3.8]{Y1} the pullback homomorphism $\pi^*$ identifies $R^*(S)$ with the subspace $\bigoplus_{i=0}^{3g-2} R^*_{(0,2i)}(\J_g)$ of $R^*(\J_g)$.
Therefore, we also obtain relations on the universal curve $\CC_g$ using the method explained above.
These relations can be pushed down to $\CM_g$ via the canonical map $\CC_g \to \CM_g$ to give relations in $R^*(\CM_g)$ as well.
All tautological relations on $\CC_g$ for $g \leq 19$ and on $\CM_g$ for $g \leq 23$ can be recovered using this method.
\end{rem}

\begin{dfn}
\label{defi: Faber Zagier brace operation}
Let $\vec{\jmath}=(j_0, j_1,\ldots)$ be an ordered set of indeterminates. Then the $\mathbb{Q}$-linear bracket operation on power series is
\[
\{x^n\}_{\vec{\jmath}} \coloneqq x^n j_n
\]
\end{dfn}
\begin{ex}
\label{example Faber Zagier A series}
Let $A(z)$ be the power series 
\[A(z)=\sum_{n=0}^\infty \frac{(6n)!}{(3n)!(2n)!}\left(\frac{z}{72}\right)^n.\]
The \emph{top Faber--Zagier relation of genus $3k-1$ for $\vec{\kappa}$} is the relation
\begin{equation}
\label{eq: FZ relation}
\left[\exp\left(
-\{\log(A)\}_{\vec{\kappa}}
\right)
\right]_{z^k}=0
\end{equation}
This relation corresponds to the empty partition in the previous section.
\end{ex}
\begin{rem}
\label{remark: rescaling FZ doesn't matter}
 Replacing $z$ with $\alpha z$ in the definition of $A$ merely multiplies the expression on the right side of Equation~\eqref{eq: FZ relation} by $\alpha^k$, so the choice of normalization factor $72$ is unimportant.
\end{rem}
\begin{thm}\label{thm:main}
Let $g=3k-1$.
The vanishing of the tautological class $p_{3,1}^{2k}$ on the universal Jacobian $\J_g$ gives the top Faber--Zagier relation.
\end{thm}
\begin{rem}
The precise statement of Theorem \ref{thm:main} is that by Remark \ref{jacobian_curve} the vanishing of $p_{3,1}^{2k}$ on $\J_g$ gives a relation on the universal curve $\CC_g$ over $\CM_g$.
We obtain a relation on $\CM_g$ by multiplying with $K$ and pushing down. 
In general this would involve a further layer of complication. 
But in our case, we will show that our relation on $\CC_g$ is in fact the \emph{pullback} of the top Faber--Zagier relation to $\CC_g$ via the projection $\CC_g \to \CM_g$.
This implies that the multiplication and pushdown procedure again yields the top Faber--Zagier relation.
\end{rem}

The remainder of the paper will be devoted to the proof of this theorem.
We will begin, in Section~\ref{section: lift}, by giving an interpretation of the Polishchuk differential operator and several related operators on $R^*(\J_g)$ and related rings in terms of punctured surfaces.
The Polishchuk operator itself is difficult to analyze directly, but some of the related operators are more amenable to enumerative combinatorics.
One such modified operator will be shown to yield the top Faber--Zagier relation.
Then in Section~\ref{section: relating c and psi}, we will show that in the case of interest, the output of the modified operator in fact coincides with the output of the Polishchuk operator.
As a roadmap, we have the following schematic chain of proportionalities and equalities:
\begin{align*}
\parbox{2.5cm}{\raggedright left hand side of~\eqref{eq: FZ relation}} &\propto \partial_-^{3k}(q_{3,1}^{2k})&&\text{Corollary~\ref{cor:fz1}}
\\&=\widehat{\pullbackmap}\circ \partial^{3k}_{c-}(q^{2k}_{3,1})|_{\rescaledkappa_0=6k-4}&&\parbox{3cm}{\raggedright Section~\ref{sec: proof of main thm} via Lemma~\ref{lem: evaluation of master series}}
\\[.75em]
&=\pullbackmap\circ \partial^{3k}_{\Poli}(q^{2k}_{3,1})|_{\rescaledkappa_0=6k-4}
&&\text{Lemma~\ref{lemma: poli versus c-}}
\\[.75em]
&\propto\parbox{4cm}{pushforward to $\mathcal{M}_g$ of the image of $p^{2k}_{3,1}$ under the Polishchuck operator}&&\parbox{3cm}{\raggedright Remarks~\ref{rem: relation to geometric operator} and~\ref{remark: relation to Yin's pullback}}
\end{align*}
which together assemble to the proof of Theorem~\ref{thm:main}.

\begin{rem}
The conclusion of Theorem~\ref{thm:main} is not itself surprising. We know from~\cite{Y1} that this method should give a relation in the tautological ring of $\mathcal{M}_g$.
According to \cite{F2} it is expected that the space of degree $k$ relations in the tautological ring of $\CM_g$ when $g=3k-1$ should be one dimensional. 
Therefore, we expect to get the unique known relation up to a scalar multiple, and in that sense the content is that the scalar multiple is not zero. Our emphasis instead is that the method of proof is new---this is a proof of concept that Yin's theory can be used effectively to extract concrete relations.
\end{rem}

\section{Lifting of the Polishchuk differential operator using punctured surfaces}
\label{section: lift}
In this section we study several polynomial algebras related to the tautological ring of the universal Jacobian $R^*(\J_g)$. 

\begin{notation}
We use the notation ${\Lambda_q}$ for the polynomial ring $\QQ[q_{i,j},\rescaledpsi]$ on $\rescaledpsi$ and variables $q_{i,j}$ with $i\equiv j\pmod 2$ and $i,j\ge 0$.
We use the notation $\widehat{{\Lambda_q}}$ for the further extension of $\Lambda_q$ by two more variables: $\widehat{\Lambda_q}={\Lambda_q}[q_{0,-2},q_{1,-1}]$.
\end{notation}
\begin{rem}
\label{remark: generators are surfaces}
We think of the generator $q_{i,j}$ as representing an orientable surface of Euler characteristic $-j$ with $i$ boundary components, and passing to the extended ring corresponds to allowing generators for the disk and the sphere as well as all other orientable surfaces of finite type.
\end{rem}
The evident inclusion and projection are maps of rings between ${\Lambda_q}$ and $\widehat{{\Lambda_q}}$. 
There is a projection $\pi$ from ${\Lambda_q}$ to $R^*(\J_g)$ defined by
\begin{align}
q_{i,j}&\mapsto \frac{i!}{2^{\frac{i+j-2}{2}}}p_{i,j}
&
\rescaledpsi&\mapsto \frac{\psi}{4}
\label{eq: change of basis}
\end{align}
(killing generators that are out of range).
The projection $\pi$ realizes $R^*(\J_g)$ as a quotient of ${\Lambda_q}$ or $\widehat{{\Lambda_q}}$.
The following differential operators will be our main players.
\begin{dfn} 
We define the following differential operators on ${\Lambda_q}$:
\label{definition: differential operators}
\begin{itemize}
\item the \emph{one-component gluing operator}
\begin{equation*}
\partial_1 \coloneqq\sum_{i,j}\binom{i}{2}q_{i-2,j}\partial_{q_{i,j}},
\end{equation*}
\item the \emph{two-component gluing operator}
\begin{equation*}
\partial_2 \coloneqq\sum_{i,j,k,l}\frac{1}{2}ik q_{i+k-2,j+l}\partial_{q_{i,j}}\partial_{q_{k,l}},
\end{equation*}
\item the \emph{one-component capping operator}
\begin{equation*}
\partial'_\psi \coloneqq\sum_{i,j}\binom{i}{2}\rescaledpsi q_{i-2,j-2}\partial_{q_{i,j}},
\end{equation*}
and
\item the \emph{two-component capping operator}
\begin{equation*}
\partial_\psi\coloneqq 	\sum_{i,j,k,l}\frac{1}{2}ik \rescaledpsi q_{i-1,j-1}q_{k-1,l-1}\partial_{q_{i,j}}\partial_{q_{k,l}}.
\end{equation*}
\end{itemize}
We also give names and notation for linear combinations of these atomic operators:
\begin{align*}
	\partial_{\Poli}&\coloneqq \partial_1 - \partial_2+\partial_\psi && \text{\emph{Polishchuk operator}}
	\\
	\partial_\pm&\coloneqq \partial_1 \pm\partial_2&&\text{\emph{gluing operators}}
	\\
	\partial_{c\pm}&\coloneqq \partial_1 \pm \partial_2 + \partial_\psi+\partial'_\psi &&\text{\emph{surface operators}}
\end{align*}
We call the gluing and surface operators \emph{positive} or \emph{negative} according to the sign of $\partial_2$.

Finally, we use the same formulas, implicitly extending the indexing of the summations for operators acting on the ring $\widehat{{\Lambda_q}}$, adding the word ``extended'' to the terminology and a hat to the notation.
\end{dfn}
\begin{rem}
\label{rem: relation to geometric operator}
Change of basis reveals that the Polishchuk operator $\partial_{\Poli}$ passes to the quotient $R^*(\J_g)$ as the Polishchuk differential operator $\mathcal{D}$ (see~Section~\ref{subsec: Lefschetz decomp}).
\end{rem}
\begin{rem}[Warning]
The differential operators which act on ${\Lambda_q}$ in Definition~\ref{definition: differential operators} are not naively compatible with their extended versions acting on $\widehat{{\Lambda_q}}$ under inclusion and projection between ${\Lambda_q}$ and $\widehat{{\Lambda_q}}$. 
\end{rem}
While the operators are not naively compatible, there is a compatibility relation related to some further quotients.
\begin{dfn}
We use the notation ${\Lambda_\rescaledkappa}$ for the ring $\mathbb{Q}[\rescaledkappa_i,\rescaledpsi]$ where $i$ varies over non-negative integers.
We use the notation $\widehat{{\Lambda_\rescaledkappa}}$ for the ring $\mathbb{Q}[\rescaledkappa_i,\rescaledpsi]$, where $i\ge -1$. 
The evaluation from $\widehat{{\Lambda_\rescaledkappa}}$ and ${\Lambda_\rescaledkappa}$ to $R^*(\CC_g)$ projects $\rescaledpsi$ to $\frac{\psi}{4}$, $\rescaledkappa_i$ to $\frac{\kappa_i}{4^i}$, and $\rescaledkappa_{-1}$ to $0$.

The \emph{inverse pullback map} $\pullbackmap:{\Lambda_q}\to {\Lambda_\rescaledkappa}$ is the $\QQ$-linear map
\begin{align}
q_{0,2n}&\mapsto
\sum_{r=0}^n \binom{n+1}{r+1}\rescaledpsi^{n-r}\rescaledkappa_r + 2^{n+1}\rescaledpsi^n.
\label{eq: pullback in q basis}
\end{align}
The \emph{extended inverse pullback map} $\widehat\pullbackmap:\widehat{{\Lambda_q}}\to\widehat{{\Lambda_\rescaledkappa}}[\rescaledpsi^{-1}]$ is 
\begin{align}
\label{extended inverse pullback map}
q_{0,2n}&\mapsto
\rescaledkappa_n + \rescaledpsi^n.
\end{align}
\end{dfn}
\begin{rem}
\label{remark: relation to Yin's pullback}
As we saw in Remark \ref{jacobian_curve} according to \cite[Corollary 3.8]{Y1} the pullback map 
$\pi^*: \CH^*(\J_g) \to \CH^*(\CC_g)$ descends to an isomorphism between the space $\bigoplus_{i=0}^{3g-2} R^*_{(0,2i)}(\J_g)$ and the tautological ring of the universal curve $\CC_g$. 
If we map our ring ${\Lambda_\rescaledkappa}$ to $R^*(\CC_g)$ via
\begin{align*}
\rescaledkappa_n&\mapsto \frac{\kappa_n}{4^n},& \rescaledpsi&\mapsto \frac{\psi}{4},
\end{align*}
then our $\pullbackmap$ descends to his ``$\pi^*$''.
Our formula~\eqref{eq: pullback in q basis} and Yin's formula (identity (3.7) in op.~cit.):
\begin{align*}
\label{eq: pullback in p basis}
p_{0,2n}&\mapsto 
\frac{1}{2^{n+1}}\sum_{r=0}^n \binom{n+1}{r+1}\psi^{n-r}\kappa_r + \psi^n
\end{align*}
differ only by our change of basis~\eqref{eq: change of basis}.
\end{rem}
\begin{lem} 
\label{lemma: poli versus c-}
Let $\widetilde{{\Lambda_q}}$ be the subalgebra of ${\Lambda_q}$ containing only $q_{i,j}$ with $i= j+2$. The following diagram commutes:
\[
\begin{tikzcd}
	\widetilde{{\Lambda_q}}
	\ar{rr}{\sum_{r=0}^\infty\partial_{\Poli}^r}
	\ar{dd}
	&&
	{\Lambda_q}\ar{r}{\pullbackmap}
	&
	{\Lambda_\rescaledkappa}\ar[dr]\\
	&&&&{\Lambda_\rescaledkappa}[\rescaledpsi^{-1}]
	\\
	\widehat{{\Lambda_q}}\ar{rr}[swap]{\sum_{r=0}^\infty\widehat{\partial}_{c-}^r}&&\widehat{{\Lambda_q}}\ar{r}[swap]{\widehat\pullbackmap}&\widehat{{\Lambda_\rescaledkappa}}[\rescaledpsi^{-1}]\ar{ur}
\end{tikzcd}
\]
\end{lem}
\begin{proof}


It suffices to check on a monomial $\nu$ in the variables $q_{i,j}$. Write $I$ for half of the total sum of all $i$ indices in the variables of $\nu$; then the only nonzero contribution comes from $\partial_{\Poli}^I$ along the top and $\widehat\partial_{c-}^I$ along the bottom. Then we are checking that
\[
\pullbackmap(\partial_{\Poli}^I\nu)=\widehat{\pullbackmap}(\widehat\partial_{c-}^I\nu)\rvert_{\rescaledkappa_{-1}=0}
.\]
It will be convenient to consider the sums involved in testing this equality as occurring over surfaces \`a la Remark~\ref{remark: generators are surfaces}, where boundary components are matched up with one another with cylinders (for applications of $\partial_1$ and $\partial_2$) or with pairs of caps (for $\partial_\psi$ and $\partial'_\psi$). Then the sum making up $\widehat\partial_{c-}^I\nu$ can be indexed over all ordered perfect pairings of boundary components, along with a choice of ``cylinder'' or ``caps'' for each pair. 
Note that the condition $i=j+2$ means that each component at the beginning is of genus zero.

Many such ordered labeled perfect pairings index the sum on the left side as well, but some are missing.
We will consider a dichotomy of two types of such missing or mismatched terms.
The first type consists of ordered perfect pairings where at some point in the iterated gluing procedure a disk ($q_{1,-1}$) arises.
The second type consists of surfaces where there is never a disk and where a pair of caps is applied to a connected surface (i.e., the operator $\widehat\partial'_\psi$ is applied).

For the first type, we note that $\widehat\partial_1$ and $\widehat\partial'_\psi$ do not act on $q_{1,-1}$, so eventually in the gluing procedure the disk ($q_{1,-1}$ term) must be glued to some other connected component. This is done either with a cylinder (via the two-component gluing operator $\widehat\partial_2$) or a pair of caps (via the two-component capping operator $\widehat\partial_\psi$); then this kind of missing term arises in pairs. 
That is, consider pairs which consist of the same operators in the same order except we swap the first application of a cylinder or caps between a disk and another surface.
Then the eventual contribution from this pair consists of $\alpha(1 - q_{0,-2}\rescaledpsi)$ for some $\alpha$. But then evaluating via $\widehat{\pullbackmap}$ yields $\widehat{\pullbackmap}(\alpha)(1-\kappa_{-1}\rescaledpsi - 1)$ which evaluates to zero when $\kappa_{-1}=0$. 

For the second type in our dichotomy, we perform a similar but more involved trick, replacing every extended one-component capping ($\widehat\partial'_\psi$) with an application of a cylinder ($\widehat\partial_1$).
The modified ordered labeled partial pairing which results indeed appears in the indexing set on the left hand side of the equation in $\partial_{\Poli}^I\nu$.
This is because no disks arise by assumption and no sphere can arise without either a disk or a pair of caps on the same connected surface.
Call this assignment (starting with an ordered labeled perfect pairing and replacing extended one-component capping with extended one-component gluing) $\zeta$.

Then it will suffice to show that for each ordered labeled perfect pairing $\Upsilon$ and the corresponding monomial $\nu_{\Upsilon}$, we have the equality
\[
\pullbackmap(\nu_\Upsilon)= \sum_{\widehat{\Upsilon}\in \zeta^{-1}(\Upsilon)} \widehat{\pullbackmap}(\widehat{\nu}_{\widehat{\Upsilon}})\biggr|_{\mathrlap{\rescaledkappa_{-1}=0}}.
\]
The connected components of the surface which arises from the gluing indexed by the ordered labeled perfect pairing $\Upsilon$ are in canonical bijection with the connected components of the surface for each $\widehat{\Upsilon}$ in the summation. 
Then it suffices to check the above equality for a single connected component. 

So now let $\Upsilon_{\partialpairing}$ be an ordered pairing (not necessarily perfect, not labeled) on the boundary components of a connected surface of genus zero with $(n+1)$ pairs.
These pairs correspond to one component gluings.
Let $\widehat{\Upsilon}_{\partialpairing}$ be a labeled version, where we label each pair either ``cylinder'' or ``caps''. Let $\alpha(\widehat{\Upsilon}_{\partialpairing})$ be the number of ``cylinder labels''.

Since we began with components of genus zero, and (extended) two-component gluings and cappings can never create genus, this eventually correspnods to a surface with no boundary components and Euler characteristic $-2n$.
Then we must only show for all $\Upsilon_{\partialpairing}$ that 
\begin{multline*}
\sum_{r=0}^n \binom{n+1}{r+1}\rescaledpsi^{n-r}\rescaledkappa_r + 2^{n+1}\rescaledpsi^n
\\=\sum \rescaledpsi^{n+1-\alpha(\widehat{\Upsilon}_{\partialpairing})}\left(\rescaledkappa_{\alpha(\widehat{\Upsilon}_{\partialpairing})-1} + \rescaledpsi^{\alpha(\widehat{\Upsilon}_{\partialpairing})-1}\right),
\end{multline*}
where the sum on the right is over all $\widehat{\Upsilon}_{\partialpairing}$ that become $\Upsilon_{\partialpairing}$ by forgetting the labeling.
Then this final equation follows from noticing that there are $2^{n+1}$ elements indexing the sum on the right hand side corresponding to a choice between cylinders and caps for each term in the pairing $\Upsilon_{\partialpairing}$.
This $2^{n+1}$ polarizes into $\binom{n+1}{r+1}$ choices of $\widehat{\Upsilon}_{\partialpairing}$ with $\alpha(\widehat{\Upsilon}_{\partialpairing})-1=r$. 
There is one term left over on the right, namely $\rescaledpsi^{n+1}\rescaledkappa_{-1}$, which evaluates to zero.
\end{proof}

We will also want to pick out a particular case.
\begin{notation}
\label{defi: the coefficient for connected surfaces with no caps}
We use the notation $\beta_{0,2k}$ to denote the coefficient of $q_{0,2k}$ in $\partial_+^{3k}(q_{3,1}^{2k})$. 
For consistency in our formulas we let $\beta_{0,2k+1}$ be zero.
\end{notation}

\begin{ex}
To illustrate the method we look at the case $g=2$ and compute the expression $\partial_+^{3}(q_{3,1}^2)$.
The monomial $q_{3,1}^2$ corresponds to two vertices and there are three leaves attached to every vertex of the graph.
Every application of the operator $\partial_1$ or $\partial_2$ corresponds to gluing two half edges.
We obtain two distinct isomorphism classes of graphs, depending on whether there are any self-gluings.
In the first case, we have no self-gluings, and the coefficient of $q_{0,2}$ is $36=(3\cdot 3)(2\cdot 2)(1\cdot 1)$.
In the second case, we have two self-gluings, and the order of gluings matters, so the coefficient is $54=3(3\cdot 3)2$, and so $\beta_{0,2}$ is $90$. See Figure~\ref{figure: first and second}


\begin{figure}[h]
\begin{tikzpicture}
\node[circle,fill=black,inner sep=2.5pt,draw] (a) at (180:1cm) {};
\node[circle,fill=black,inner sep=2.5pt,draw] (b) at (0:1cm) {};
\draw[thick] (a) edge[bend left] (b);
\draw[thick] (a) edge (b);
\draw[thick] (a) edge[bend right] (b);
\end{tikzpicture}
\qquad{}
\qquad{}
\qquad{}
\begin{tikzpicture}
\node[circle,fill=black,inner sep=2.5pt,draw] (a) at (180:1cm) {};
\node[circle,fill=black,inner sep=2.5pt,draw] (b) at (0:1cm) {};
\draw[thick] (1.5,0) circle (.5cm);
\draw[thick] (-1.5,0) circle (.5cm);
\draw[thick] (a) edge (b);
\end{tikzpicture}
\caption{The graphs corresponding to $\partial_2^3$ and $\partial_2\partial_1^2$}
\label{figure: first and second}
\end{figure}
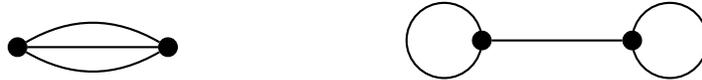



\end{ex}
Our next goal is to count $\partial_-^{3k}(q_{3,1}^{2k})$ in order to relate it to the top Faber--Zagier relation. 
It is too hard to directly obtain an explicit closed form for the coefficients involved in this expression, so we perform a kind of trick.
First we calculate directly a particular evaluation of $\partial_+^{3k}(q_{3,1}^{2k})$, which is easier both because the positive gluing operator is easier than the negative and because we're only interested in a special case.
Then we formally express $\partial_\pm^{3k}(q_{3,1}^{2k})$ in terms of the (non-explicit) coefficients $\beta_{0,2k}$.
Finally, we use the calculation in the special case to get an explicit formula for our case of actual interest.

Let us turn to the positive gluing operator. 
It behooves us to further investigate the metaphor of Remark~\ref{remark: generators are surfaces}.

\begin{dfn}
Let $\widehat{{\Lambda_\Sigma}}$ be the $\QQ$-vector space spanned by homeomorphism classes of (possibly disconnected) orientable surfaces with finite (possibly empty) labeled boundary, and let ${\Lambda_\Sigma}$ be the subspace spanned only by those surfaces with nonpositive Euler characteristic. 

Let the gluing operator $\partial_{\gluing}$ on $\widehat{{\Lambda_\Sigma}}$ take a surface $\Sigma$ to the sum of all surfaces obtained by gluing two boundary components of $\Sigma$ together.
Let the capping operator $\partial_{\capping}$ on $\widehat{\Lambda_\Sigma}$ take a surface $\Sigma$ to the sum of all surfaces obtained by capping off two distinct boundary components of $\Sigma$ with disks.
\end{dfn}
There are linear evaluations $\widehat\rho:\widehat{{\Lambda_\Sigma}}\to \widehat{{\Lambda_q}}$ and $\rho:{\Lambda_\Sigma}\to {\Lambda_q}$ which take disjoint unions to products and connected surfaces with $i$ boundary components and Euler characteristic $-j$ to $q_{i,j}$. 


\begin{lem} 
\label{lem: intertwining surface operators}
The map $\rho$ intertwines 
\begin{enumerate}
\item the operators $\partial_{\gluing}$ on ${\Lambda_\Sigma}$ and $\partial_+$ on ${\Lambda_q}$; that is,
\[
\rho \partial_{\gluing} = \partial_+\rho, 
\]
and
\item the operators $\partial_{\gluing}+\partial_{\capping}$ on $\widehat{{\Lambda_\Sigma}}$ and $\partial_{c,+}|_{\rescaledpsi=1}$ on $\widehat{{\Lambda_q}}$, i.e.,
\[
\widehat{\rho}(\partial_{\gluing}+\partial_{\capping}) = \ev_{\rescaledpsi=1} \partial_{c_+} \widehat{\rho}.
\]
\end{enumerate}
\end{lem}
\begin{proof}
Because gluing preserves Euler characteristic, the operator $\partial_{\gluing}$ restricts as follows
\[
\begin{tikzcd}
\Lambda_\Sigma\dar[hookrightarrow]\rar[dashed]{\partial_{\gluing}}&\Lambda_\Sigma\dar[hookrightarrow]
\\
\widehat{\Lambda_\Sigma}\rar[swap]{\partial_{\gluing}} & \widehat{\Lambda_\Sigma}
\end{tikzcd}
\]
so the first statement makes sense.

Applying $\partial_1$ on the right (in either case) corresponds to gluing two boundary components of the same connected surface, while applying $\partial_2$ corresponds to gluing boundary components of two different connected surfaces. 
The coefficient $\binom{i}{2}$ corresponds to the choice of two boundary components of a surface with $i$ boundary components; the coefficient $ik$ corresponds to choosing one boundary component each from surfaces with $i$ and $k$ respectively.
The half is there because the formula is symmetric and otherwise would count each pair twice.
Similarly, the one-component and two-component capping operators correspond to the operations that cap two boundary components of a single connected component or of two distinct connected components, respectively.
\end{proof}
\begin{cor}\label{cor: count of trees}
We have the following identity:
\[\partial_+^{3k}q_{3,1}^{2k}\biggr|_{q_{0,n}=1}=\frac{(6k)!}{8^k}
\] 
(the evaluation happening for all $n$).
\end{cor}
\begin{proof}
Let $\Sigma$ be the disjoint union of $2k$ pairs of pants. 
Then
\[\partial_+^{3k}q_{3,1}^{2k}=\partial_+^{3k}\rho(\Sigma)=\rho \partial_{\gluing}^{3k}(\Sigma).\]
But $\partial_{\gluing}^{3k}(\Sigma)$ is a sum over the perfect pairings on the $6k$ boundary components of $\Sigma$ along with an ordering of the $3k$ pairs of the perfect pairing. Then there are 
\[
(6k-1)!!(3k)!=\frac{(6k)!(3k)!}{(3k)!2^{3k}}=\frac{(6k)!}{8^k}
\]
of these.
\end{proof}

Let $\mathfrak{P}(n)$ denote the unordered partitions of the set $\{1,\ldots,n\}$ (see Appendix~\ref{app:series} for details on our notation).
\begin{lem}\label{lem: D as exp-type sum}
We have the following identity:
\begin{equation}\label{equation: D as exp}
 \sum_{k=0}^\infty 
z^{2k} 
\frac{\partial_\pm^{3k}(q_{3,1}^{2k})}{(3k)!(2k)!}
=
\exp\left(
\sum_{n=1}^\infty\frac{\pm \beta_{2n}}{(3n)!(2n)!}q_{0,2n}z^{2n}
\right).
\end{equation}

\end{lem}
\begin{proof}
Since $\partial_{\pm}$ acts on the bigraded ring ${\Lambda_q}$ by lowering the first grading by two and preserving the second grading, necessarily $\partial_{\pm}^{3k}(q_{3,1}^{2k})$ is of the form
\[
\sum_{\mathfrak{p}\in \mathfrak{P}(2k)}\alpha_{\mathfrak{p}} \prod_{\mathfrak{b}_i\in \mathfrak{p}} q_{0,|\mathfrak{b}_i|}
\]
for some coefficient $\alpha_{\mathfrak{p}}$ depending on the partition $\mathfrak{p}$ (the partition keeps track of which copies of $q_{3,1}$ have been glued together). 
Each $|\mathfrak{b}_i|$ must be even for the coefficient to be non-zero, so we may assume $|\mathfrak{b}_i|=2k_i$ now (i.e., $\alpha_\mathfrak{p}=0$ if $\mathfrak{p}$ has an odd length block).
Each individual summand must arise by applying the two-component gluing operator $\partial_2$ precisely $2k_i-1$ times and the one-component gluing operator $\partial_1$ precisely $k_i+1$ times in some order to the monomial $q_{3,1}^{2k_i}$.
The sum of all the terms that arise in this way is then $\pm \beta_{2k_i}q_{0,2k_i}$ by Definition~\ref{defi: the coefficient for connected surfaces with no caps}. 
The sign is negative for the negative gluing operator because $\partial_2$ is necessarily applied an odd number of times. 
To get the coefficient $\alpha_\mathfrak{p}$, we need to combine these calculations over the blocks $\mathfrak{b}_i$ in the partition $\mathfrak{p}$.
This entails distributing the $3k$ applications of $\partial_{\pm}$ into individual (not necessarily contiguous) blocks of size $3k_i$. The number of ways of doing that is
\[
\frac{(3k)!}{\prod_{\mathfrak{b}_i\in \mathfrak{p}}(\frac{3}{2}|\mathfrak{b}_i|)!},
\]
so we get 
\[
\partial_{\pm}^{3k}(q_{3,1}^{2k})
=\sum_{\mathfrak{p}\in \mathfrak{P}(2k)}(3k)!\prod_{\mathfrak{b}_i\in \mathfrak{p}}\frac{\pm \beta_{|\mathfrak{b}_i|}}{(\frac{3}{2}|\mathfrak{b}_i|)!}q_{0,|\mathfrak{b}_i|}.
\]
By the exponential compositional formula (reviewed as Corollary~\ref{cor: exponential compositional formula}), the exponential generating function for $\frac{1}{(3k)!}\partial_{\pm}^{3k}(q_{3,1}^{2k})
$ is the formal exponential of the exponential generating function for $\frac{\pm \beta_{|\mathfrak{b}_i|}}{(\frac{3}{2}|\mathfrak{b}_i|)!}q_{0,|\mathfrak{b}_i|}$.
These are precisely the left and right sides of Equation~\eqref{equation: D as exp}.
\end{proof}

\begin{cor}\label{cor:fz1}
Write $\vec{q}=(q_{0,0},q_{0,2},q_{0,4},\ldots)$. Then
\[
\partial_-^{3k}(q_{3,1}^{2k})=0
\]
is the top Faber--Zagier relation of genus $3k-1$ for $\vec{q}$.
\end{cor}
\begin{proof}
First, evaluate the positive version of Equation~\eqref{equation: D as exp} at $q_{0,j}=1$. 
For the right side use Corollary~\ref{cor: count of trees}.
Taking formal logarithms we then get the equation
\[
\sum_{n=1}^\infty \frac{\beta_{2n}}{(3n)!(2n)!}z^{2n} =
\log\left( \sum_{n=0}^\infty\frac{(6n)!}{(3n)!(2n)!}\left(\frac{z^2}{8}\right)^n\right)=\log(A(9z^2)) 
\]
where the rightmost expression uses the series of Example~\ref{example Faber Zagier A series}.

Then the vanishing of the $z^{2n}$ term of the right hand side of Equation~(\ref{equation: D as exp}) for the negative gluing operator is equivalent to the vanishing of the $z^{2n}$ term of the right-hand side, i.e., the $z^n$ term (dividing powers of $z$ by two in the series) of 
\begin{align*}
\exp\left(
\sum_{n=1}^\infty\frac{- \beta_{2n}}{(3n)!(2n)!}q_{0,2n}z^n
\right)
&=
\exp\left(-\left\{\sum_{n=1}^\infty \frac{\beta_{2n}}{(3n)!(2n)!}z^n\right\}_{\vec{q}} \right)
\\
&=
\exp
\left(
-\left\{
\log(A(9z))
\right\}_{\vec{q}}
\right)
\end{align*}
(see Definition~\ref{defi: Faber Zagier brace operation} for notation).
By Remark~\ref{remark: rescaling FZ doesn't matter} the factor of $9$ does not affect the equation, and so the vanishing of the $z^n$ term here coincides with the relation of Equation~\ref{eq: FZ relation}.
\end{proof}

\section{The cancellation of contributions from \texorpdfstring{$\psi$}{psi}-classes}
\label{section: relating c and psi}

The goal of this section is to complete the proof of Theorem~\ref{thm:main}, which says that the vanishing of the tautological class $p_{3,1}^{2k}$ on $\J_g$ gives the top Faber--Zagier relation on the tautological ring of $\mathcal{M}_g$ by the method explained in Remark \ref{jacobian_curve}. 

According to Corollary \ref{cor:fz1} we know that the $3k$-fold application of the negative gluing operator $\partial_{-}^{3k}$ to the monomial $p_{3,1}^{2k}$ gives the top Faber--Zagier relation of genus $3k-1$. 
We also know via Lemma~\ref{lemma: poli versus c-} that the relation arising from the Polishchuk operator coincides with that arising from the $3k$-fold application of the negative surface operator $\partial_{c-}^{3k}$ to the same monomial.

However, the negative gluing and negative surface operators differ. 
The difference is given by the operator $\partial_{\psi}+\partial'_{\psi}$. 
To complete the proof, in this section we will show that all contributions of $\psi$ classes cancel with one other when we consider $\widehat{\pullbackmap}\circ \partial_{c-}^{3k}(p_{3,1}^{2k})$. 

\subsection{Enumerative combinatorics for surface gluing}
Our current and final goal is to perform the enumerative combinatorics for the application $\widehat{\pullbackmap}\circ \partial_{c-}^{3k}(p_{3,1}^{2k})$.
\begin{lem}
\label{lem: calculation of dc- on the qs}
The expression $\partial_{c-}^{3k}q_{3,1}^{2k}$ can be written as a sum over graphs as follows.
\[
\partial_{c-}^{3k}q_{3,1}^{2k}
=
(3k)!\sum_{r=0}^{3k}(2r-1)!! \rescaledpsi^{r} \sum_{\chi(\Gamma)=2r-2k}
\prod_{\Gamma_c}-q_{0,-\chi(\Gamma_c)},
\]
where $\Gamma$ runs over isomorphism classes of possibly disconnected ordered trivalent graphs of Euler characteristic $2r-2k$ with precisely $2k$ vertices and $2r$ leaves and $\Gamma_c$ runs over connected components of $\Gamma$.
\end{lem}
\begin{proof}
This essentially follows from Lemma~\ref{lem: intertwining surface operators}, which is written in terms of $\partial_{c+}$ and surfaces involves the evaluation at $\rescaledpsi=1$.
However, of the four constituent operators of the negative surface operator $\partial_{c-}$, only the two-component gluing operator changes the number of connected components (always reducing it by one) and only the capping operators change the Euler characteristic (always reducing it by two).
Therefore we can recover the overall sign of a term as well as its power of $\rescaledpsi$ from the combinatorics of the surface. 
Then isomorphism classes of closed orientable surfaces equipped with a fixed decomposition into pairs of pants and disks in which each connected component contains at least one pair of pants are in natural bijection with isomorphism classes of trivalent graphs with leaves---vertices correspond to pants and leaves to disks. 
The Euler characteristic of a surface made of gluing $2k$ pairs of pants and $2r$ disks is  $2r-2k$, while the Euler characteristic of a trivalent graph with $2k$ vertices and $2r$ leaves is $r-k$.
The coefficients $(3k)!$ and $(2r-1)!!$ come from choosing an order for the gluings and for choosing a perfect pairing between the caps for the capping operators.
\end{proof}
Recall the extended inverse pullback map,~\eqref{extended inverse pullback map}.
We want to apply this to our calculation of $\partial_{c-}^{3k}q_{3,1}^{2k}$ from Lemma~\ref{lem: calculation of dc- on the qs}.
The extended inverse pullback map takes the term $-q_{0,2n}$ corresponding to the connected component $\Gamma_c$ of Euler characteristic $-2n$ to the sum $-\rescaledkappa_n-\rescaledpsi^n$.
We will reorganize this application of the extended inverse pullback map into a summation over powers of $\rescaledpsi$, which will turn out eventually to have little reliance on the monomial in $\rescaledkappa_n$ variables.

For this purpose we introduce several generating functions so that we can perform the computation in formal series. See appendix~\ref{app:trees} for our conventions on graphs (in particular the definition of an \emph{ordered} graph).
\begin{dfn}
Let $\mathcal{G}_+(n,m)$ be the set of isomorphism classes of (possibly disconnected) trivalent ordered graphs with $n$ vertices and $m$ leaves, such that each connected component has positive Euler characteristic. 
Similarly define $\mathcal{G}_0(n,m)$ and $\mathcal{G}_-(n,m)$ with the indicated Euler characteristic restrictions on each connected component.

We use the superscripts $c$ and $\lf$ to further restrict to graphs that are connected and leaf free, respectively.

Let $G_+(x,y)$ denote the following generating function for $\mathcal{G}_+(n,m)$:
\[
G_+(x,y)=\sum_{m,n\ge 0}\#\mathcal{G}_+(n,m)\frac{x^n}{n!}y^m.
\]
Define $G_0(x,y)$, $G_-(x,y)$, and the connected variants similarly. 
The leaf-free variant is obtained from the general formula by evaluating $y=0$ and we will think of it as a single variable series instead.
\end{dfn}
So $\mathcal{G}_+$ consists of forests, $\mathcal{G}_0$ of disjoint unions of ``hairy loops'', and $\mathcal{G}_-$ of ``everything else''. 
We will also need to count another kind of tree.
\begin{dfn}\label{trees}
Let $\mathcal{T}_{rr}(n,m)$ be the set of isomorphism classes of ordered trivalent trees $T$ with $n$ trivalent vertices and $m+2$ leaves equipped with an ordered pair of distinct distinguished leaves, called the \emph{roots}.
Let $T_{rr}(x,y)$ denote the following generating function:
\[
T_{rr}(x,y)=\sum_{m,n\ge 0} \#\mathcal{T}_{rr}(n,m)\frac{x^n}{n!}y^{m}.
\]
\end{dfn}

Now we will combine these generating functions into a single generating function in five variables which will keep track of a complicated weighted count of graphs. 
The behavior of these series will give us a key to relate the desired quantity $\widehat{\pullbackmap}\ \partial^{3k}_{c-}(q^{2k}_{3,1})$ to the simpler expression
$\partial^{3k}_{-}(q^{2k}_{3,1})$ analyzed in Corollary~\ref{cor:fz1}.

\begin{dfn}
\label{definition: master series}
The \emph{master series} $\masterseries(x,y,z,w,u)$ is the following formal series:
\[
\masterseries(x,y,z,w,u)= \frac{e^{zT_{rr}(x,y)}}{e^{wG_0^c(x,y)}G_+(x,y)G^{\lf}_-(xu)}.
\]
\end{dfn}
The following technical lemma providing a kind of evaluation of the master series constitutes the promised key.
\begin{lem}
\label{lem: evaluation of master series}
Given a non-negative integer $n$, construct a one-variable series $\simplifiedmasterseries(x)$ by performing the following $\mathbb{R}$-linear substitution of monomials on $\masterseries$:
\begin{multline*}
x^{n_1}y^{n_2} z^{n_3} w^{n_4} u^{n_5} \mapsto 
\\\begin{cases}
0 & \!n_2\text{ odd};
\\
(n_2-1)!!n_3!\binom{\frac{3}{2}n_5+3n}{n_3} (3n_1+6n-3)^{n_4} x^{n_1}
 & \!n_2\text{ even.}
\end{cases}
\end{multline*}
Then $\simplifiedmasterseries$ is well-defined and equal to $1$ for all $n$.
\end{lem}
We will defer the proof of this lemma, which is a somewhat involved exercise in formal series, to Appendix~\ref{appendix: key lemma proof}, and meanwhile use it to prove the main theorem.
\subsection{Proof of the main theorem}
\label{sec: proof of main thm}
We are now ready to complete the proof of Theorem~\ref{thm:main}. As per the discussion at the beginning of the section, it is enough to show
\[
\widehat{\pullbackmap}\circ \partial^{3k}_{c-}(q^{2k}_{3,1})|_{\rescaledkappa_0=6k-4}= \partial^{3k}_{-}(q^{2k}_{3,1})
\]
for all $k$. 
The evaluation at $6k-4$ corresponds to the identity $\kappa_0=2g-2$ noted at the beginning of Section~\ref{section: FZ relations} and the assignment $g=3k-1$ of Example~\ref{example Faber Zagier A series}.

\begin{lem}\label{lem1}
Let $g=3k-1$ for $k \geq 1$. Then
\begin{multline*}\label{eq1}
\widehat{\pullbackmap}\circ\partial_{c-}^{3k}q_{3,1}^{2k}|_{\rescaledkappa_0=6k-4}
=
(3k)!\sum_{r=0}^{3k}(2r-1)!!  
\\\sum_{
\mathclap{
\chi(\Gamma_\rescaledkappa\sqcup \Gamma_0\sqcup \Gamma_{\pm})=r-k
}
}\rescaledpsi^{r-\chi(\Gamma_{\pm})}(3-6k)^{\#\pi_0(\Gamma_0)}(-1)^{\#\pi_0(\Gamma_\pm)}
\prod_{\Gamma_c\subset \Gamma_\rescaledkappa}-\rescaledkappa_{-\chi(\Gamma_c)},
\end{multline*}
where the summation on the second line is over isomorphism classes of ordered trivalent graphs of Euler characteristic $r-k$ with $2k$ vertices and $2r$ leaves equipped with a decomposition as a disjoint union into
\begin{enumerate}
\item a graph $\Gamma_\rescaledkappa$ all of whose connected components have negative Euler characteristic,
\item a graph $\Gamma_0$ all of whose connected components have zero Euler characteristic, and
\item a graph $\Gamma_{\pm}$ all of whose connected components have nonzero Euler characteristic
\end{enumerate}
and the product is over connected components of $\Gamma_{\rescaledkappa}$.
\end{lem}
\begin{proof}
By Lemma~\ref{lem: calculation of dc- on the qs}, we can write $\partial^{3k}_{c-}(q^{2k}_{3,1})$ as a sum over graphs.
Then applying $\widehat{\pullbackmap}$ corresponds, for such a graph $\Gamma$, to choosing some subgraph $\Gamma_\rescaledkappa$ consisting of a collection of connected components of $\Gamma$ to evaluate via $q_{0,2n}\mapsto\rescaledkappa_n$ and the complementary subgraph $\Gamma\setminus \Gamma_\rescaledkappa$ to evaluate via $q_{0,2n}\mapsto \rescaledpsi^n$. 
This needs a little modification because of the exceptional values when $n=-1$ and $n=0$. 
That is, when $n=-1$, we evaluate $\rescaledpsi^{-1}+\rescaledkappa_{-1}$ to $\rescaledpsi^{-1}$ and when $n=0$ we evaluate $\rescaledpsi^{0}+\rescaledkappa_0$ to $6n-3$ as described above.
So we should separate out cases according to the sign of the Euler characteristic of connected components, and only include connected components with negative Euler characteristic in $\Gamma_\rescaledkappa$.
\end{proof}

We would like to more or less ``hold $\Gamma_\rescaledkappa$ fixed'' in Lemma \ref{lem1} and show that the coefficients are individually zero except for the trivial power of $\rescaledpsi$, i.e., the case $r-\chi(\Gamma_{\pm})=0$. 
But there is a relationship between $\Gamma_\rescaledkappa$ and the summation index $r$ which makes this slightly awkward.
To get around this issue, we will use the decomposition of Appendix~\ref{app:trees}.
This decomposition starts from a conneced trivalent graph $\Gamma$ of negative Euler characteristic and outputs a connected leaf-free trivalent graph of the same Euler characteristic, called the \emph{core} of $\Gamma$, along with a collection of trivalent doubly rooted trees, the \emph{insertion forest}.

\begin{lem}
\label{lem: insertion tree simplification lemma}
Let $g=3k+1$ for $k\ge 1$. Then
\begin{multline*}
\widehat{\pullbackmap}\circ\partial_{c-}^{3k}q_{3,1}^{2k}
=
(3k)!
\sum_{\Gamma_\rescaledkappa'}\left(\prod_{\Gamma_c'\subset \Gamma_\rescaledkappa'} -\rescaledkappa_{-\chi(\Gamma_c)}\right)
\sum_{r=0}^{3k}(2r-1)!!
\\
\sum_{\mathclap{\Gamma_-',\Gamma_0,\Gamma_+,\Gamma_{\ins}\subset \Gamma}}\rescaledpsi^{r-\chi(\Gamma_+)-\chi(\Gamma_-')}\binom{\frac{3}{2}(\#V_\rescaledkappa'+\#V_-')  }{\#\pi_0(\Gamma_{\ins})}(3-6k)^{\#\pi_0(\Gamma_0)}(-1)^{\#\pi_0(\Gamma_-'\sqcup \Gamma_+)}
\end{multline*}
where
\begin{enumerate}
\item $\Gamma_\rescaledkappa'$ and $\Gamma_-'$ vary over leaf-free ordered trivalent graphs, 
\item $\Gamma_0$ varies over ordered trivalent graphs where every component has Euler characteristic zero, 
\item $\Gamma_+$ varies over ordered trivalent forests, 
\item $\Gamma_{\ins}$ varies over ordered trivalent forests where each connected component is given an ordered pair of distinct roots, such that the total number of vertices of all five graphs is $2k$ and the total number of non-root leaves is $2r$, and
\item all of the orders are subordinate to an order on the disjoint union $\Gamma$ of all of the five graphs.
\end{enumerate}
\end{lem}
\begin{proof}
Beginning with the situation of Lemma~\ref{lem1}, we can further decompose $\Gamma_{\pm}$ into $\Gamma_-$ (its components with negative Euler characteristic) and $\Gamma_+$ (its components with positive Euler characteristic), and perform the reduction procedure of the appendix on $\Gamma_c$ and $\Gamma_-$. 
Then we have cores $\Gamma_\rescaledkappa'$ of $\Gamma_\rescaledkappa$ and $\Gamma_-'$ of $\Gamma_-)$ and insertion forests for them. 
If we like, we can think of a single insertion forest $\Gamma_{\ins}$ for the disjoint union $\Gamma_{\rescaledkappa}'\sqcup \Gamma_-'$.
Then the number of graphs with insertion forest $\Gamma_{\ins}$ and core $\Gamma'$ is $\binom{E(\Gamma')}{\#\pi_0(\Gamma_{\ins})}$.
Since the core is leaf-free and trivalent, we have $3V=2E$, so we can write the edge count as $\frac{3}{2}$ the vertex count.
\end{proof}

Now given the equation of Lemma~\ref{lem: insertion tree simplification lemma}, write $n_1$ for the total number of vertices of $\Gamma_-'\sqcup \Gamma_0\sqcup \Gamma_+\sqcup \Gamma_{\ins}$ and write $2n$ for the total number of vertices of $\Gamma_\rescaledkappa'$. 
Note that $\chi(\Gamma_+)+\chi(\Gamma_-')=n_1-r$. 
Another simplification comes from noticing that $r\le 3k$ is not a necessary constraint to specify because it follows automatically from graph combinatorics, which can be seen as follows.
Leaf-free graphs, connected trivalent graphs of Euler characteristic zero, and doubly rooted trees all have at least as many vertices as non-root leaves. 
On the other hand, a tree has $2$ more leaves than vertices.
Then for a fixed $k$, the maximum number of non-root leaves that can occur with a set of graphs $(\Gamma_\rescaledkappa',\Gamma_-',\Gamma_0,\Gamma_+,\Gamma_{\ins})$ as above is when $\Gamma_{\ins}$ has $2k$ connected components, each a single vertex. 
In this case there are $6k=2(3k)$ leaves so $r>3k$ is not possible.

Then we can separate out a factor of $(3k)!\binom{2k}{2n}\sum\prod -\rescaledkappa_{-\chi(\Gamma_c)}$ where the binomial coefficient comes from choosing the which $2n$ vertices of the $2k$ lie in $\Gamma_{\rescaledkappa}'$. The coefficient of this term is then:
\begin{multline*}
\sum_{r=0}^{\infty}(2r-1)!!
\sum_{{\Gamma_-',\Gamma_0,\Gamma_+,\Gamma_{\ins}}}
\binom{n_1}{\#V_-',\#V_0,\#V_+,\#V_{\ins}}\rescaledpsi^{n_1}\binom{3n+\frac{3}{2}\#V_-'}{\#\pi_0(\Gamma_{\ins})}\cdot
\\(3-3(n_1+2n))^{\#\pi_0(\Gamma_0)}(-1)^{\#\pi_0(\Gamma_-'\sqcup \Gamma_+)}
.\end{multline*}
Now the sum is being taken over tuples of ordered trivalent graphs $(\Gamma_-',\Gamma_0,\Gamma_+,\Gamma_{\ins})$ with Euler characteristic restrictions and extra root data for $\Gamma_{\ins}$ but without any ambient graph $\Gamma$.
This sum in turn is obtained from the formal series in $x$, $y$, $z$, $w$, and $u$
\begin{equation}
\label{graphed master poly equation}
\sum_{\Gamma_-',\Gamma_0,\Gamma_+,\Gamma_{\ins}} \frac{n_1!(-1)^{\#\pi_0(\Gamma_-'\sqcup \Gamma_0\sqcup\Gamma_+)}}{\#V_-'!\#V_0!\#V_+!\#V_{\ins}!}x^{n_1}y^{2r}z^{\#\pi_0(\Gamma_{\ins})}w^{\#\pi_0(\Gamma_0)}u^{\#V_-'}
\end{equation}
by linearly replacing $x^{n_1}y^{n_2}z^{n_3}w^{n_4}u^{n_5}$ as in Lemma~\ref{lem: evaluation of master series} and evaluating at $x=\rescaledpsi$.

\begin{lem}\label{lem3}
The coefficients of $x^{n_1}y^{n_2}z^{n_3}w^{n_4}u^{n_5}$ in~\eqref{graphed master poly equation} and in the master series $\masterseries(x,y,z,w,u)$ of Definition~\ref{definition: master series} agree up to a nonzero scalar multiple depending only on $n_1$ which is $1$ for $n_1=0$.
\end{lem}
\begin{proof}
Rescale the coefficients of~\eqref{graphed master poly equation} by dividing by $n_1!$. 
Then we can decompose the rescaled series, using the fact that $\Gamma_-'$ has no leaves, $\Gamma_0$ and $\Gamma_{\ins}$ have as many non-root leaves as vertices, and $\Gamma_+$ has two leaves more than its number of vertices:
\begin{multline*}
\sum_{\Gamma_-',\Gamma_0,\Gamma_+,\Gamma_{\ins}} 
\frac{(-1)^{\#\pi_0(\Gamma_-'\sqcup \Gamma_0\sqcup\Gamma_+)}}{\#V_-'!\#V_0!\#V_+!\#V_{\ins}!}
x^{n_1}y^{2r}z^{\#\pi_0(\Gamma_{\ins})}w^{\#\pi_0(\Gamma_0)}u^{\#V_-'}
\\
=
\sum_{\Gamma_-'}\frac{(-1)^{\#\pi_0(\Gamma_-')}}{\#V_-'!}(xu)^{\#V_-'}
\sum_{\Gamma_0}\frac{(-w)^{\#\pi_0(\Gamma_0)}}{\#V_0!}(xy)^{\#V_0}
\\
\sum_{\Gamma_+}\frac{(-y^2)^{\#\pi_0(\Gamma_+)}}{\#V_+!}(xy)^{\#V_+}
\sum_{\Gamma_{\ins}} \frac{z^{\#V_{\ins}}}{\#V_{\ins}!}(xy)^{\#V_{\ins}}.
\end{multline*}
These are all exponential generating functions for the appropriate types of ordered trivalent graphs, so this product is
\[
\exp(-G_-^{\lf,c}(xu))\exp(-wG_0^c(x,y))\exp(-G_+^c(x,y))\exp(zT_{rr}(x,y)),\ 
 \] 
which is the master series.
\end{proof}
\begin{proof}[Proof of Theorem~\ref{thm:main}]
By Lemma~\ref{lem: evaluation of master series}, for any $n$ the coefficient of $x^{n_1}$ in $\simplifiedmasterseries(x)$ is $1$ if $n_1=0$ and $0$ otherwise.
Then by Lemma~\ref{lem3}, the same is true for the ``evaluation'' \`a la Lemma~\ref{lem: evaluation of master series} of the expression~\eqref{graphed master poly equation}.
This implies that for each choice of $\Gamma'_{\rescaledkappa}$, the coefficient in $\widehat{\pullbackmap}\circ\partial_{c-}^{3k}q_{3,1}^{2k}|_{\rescaledkappa_0=6k-4}$ of the corresponding expression in $\rescaledkappa_j$ variables has no dependence on $\rescaledpsi$.

But how does the $\rescaledpsi^0$ term of the expression $\widehat{\pullbackmap}\circ\partial_{c-}^{3k}q_{3,1}^{2k}|_{\rescaledkappa_0=6k-4}$ differ from $\partial_{-}^{3k}q_{3,1}^{2k}$?
The operator $\partial_{c-}$ differs from $\partial_-$ by terms with a coefficient of $\rescaledpsi$. 
On the other hand, the operator $\widehat{\pullbackmap}$ differs from the identity by terms with positive powers of $\rescaledpsi$ and special terms with power $\rescaledpsi^0$ for $q_{0,0}$ and $\rescaledpsi^{-1}$ for $q_{0,-2}$. 
But to arrive at $q_{0,0}$ we must have a surface of genus $1$, implying at least one cap, and to arrive at $q_{0,-2}$ we must have a surface of genus $0$, implying at least three caps. 
Since caps only arise (in pairs) from the application of $\partial_\psi$ and $\partial'_\psi$, this means that for terms containing $q_{0,0}^iq_{0,-2}^j$ there must have been at least $\frac{3j+i}{2}$ such operators, yielding a total power of $\rescaledpsi$ of at least $\frac{3j+i}{2}-j$. This is greater than zero unless $i=j=0$.
But in the $i=j=0$ case we are looking at summands where we have applied only $\partial_-$ and then avoided the two special cases of $\widehat{\pullbackmap}$. 
Then the coefficient of $\rescaledpsi^0$ in $\widehat{\pullbackmap}\circ\partial_{c-}^{3k}q_{3,1}^{2k}|_{\rescaledkappa_0=6k-4}$ is exactly $\partial_{-}^{3k}q_{3,1}^{2k}$.
\end{proof}

\section{Final remarks}

The following conjecture was proposed by Yin:

\begin{conj}
Every relation in the tautological ring of $\CC_g$ comes from a relation on the universal Jacobian $\J_g$.
\end{conj}
Yin further conjectured that the $\fsl_2$ action is the only source of all tautological relations.
For more details and precise statements we refer to \cite[Conjecture 3.19]{Y1}.
In a similar way one can use relations on the universal Jacobian and produce relations on all powers of the universal curve.
Conjecturally these relations are enough.
In \cite{tavakol_invariants} it is proved that all tautological relations on $\CC_2^n$ can be obtained from the universal Jacobian.
From the results in \cite{PTY} the same is true for the moduli spaces $\CC_g^n$ for $g=3,4$ and $n \in \N$ as well as $\CC_5^n$ when $n \leq 7$. 
By the results in \cite{P3, tavakol_conjectural} this procedure leads to a conjectural description of the space of relations on the moduli space $\CM_{g,n}^{rt}$ of stable $n$-pointed curves of genus $g$ with rational tails.
The analogous version of Faber--Zagier relations on the moduli space $\M_{g,n}$ are introduced by Pixton \cite{Pi}.
These relations are conjectured to be all tautological relations.
There is a natural way to restrict these relations on the space $\CM_{g,n}^{rt}$ and push them forward to $\CC_g^n$ via the morphism $\CM_{g,n}^{rt}\to \CC_g^n$ which contracts all rational components. 
It is a reasonable hope that the method studied in this article can be used to prove more Faber--Pixton--Zagier relations on $\CC_g^n$ from the universal Jacobian. 
This would give a partial positive answer to \cite[Conjecture 3.24]{tavakol_conjectural}.

\appendix

\section{Trivalent graphs and cores}
\label{app:trees}
In this appendix, we describe a structure theorem for trivalent graphs (i.e., all vertices are of valence $1$ or $3$---but the graph may have leaves) which realize such graphs as being built by inserting a collection of trees into a closed trivalent core graph.

\subsection{Conventions on graphs}
We consider a graph $\Gamma$ as being a tuple $(V,H,\tau,\iota)$ where $V$ is the set of \emph{vertices}, $H$ is the set of \emph{half-edges}, $\tau:H\to V$ is the target map, and $\iota:H\to H$ is an involution. 
An \emph{edge} is a non-fixed orbit of $\iota$, and the set of edges is denoted $E$ or $E(\Gamma)$.
A \emph{leaf} is a fixed orbit of $\iota$.
A graph is \emph{leaf-free} if it has no leaves.
If $h$ is a half-edge then $\tau(h)$ is the vertex of $h$.
If $v$ is a vertex then $\tau^{-1}(v)$ is the set of half-edges of $v$. 
The \emph{valence} of $v$ is the cardinality $|\tau^{-1}(v)|$; in particular \emph{trivalent} means valence $3$.
The geometric realization $|\Gamma|$ of a graph $\Gamma$ is the topological space
\[
(([0,1]\times H) \sqcup V )/ \left\{(0,h)\sim \tau(h), (t,h)\sim (1-t,\iota(h))\right\}.
\]
A \emph{connected component} of a graph $\Gamma$ is a subgraph whose geometric realization is a connected component of $|\Gamma|$.
A graph is a \emph{tree} if its geometric realization is simply connected and a \emph{forest} if its connected components are trees.
The \emph{Euler characteristic} of $\Gamma$ is $\chi(\Gamma)=|V|-|E|$.
\begin{lem}
\label{lem: euler trivalent}
The Euler characteristic of a trivalent graph is $\frac{|L|-|V|}{2}$ where $L$ is the set of leaves.
\end{lem}
\begin{proof}
We know $3|V|=|H|$ and $|H|=2|E|+|L|$ so $|E|=\frac{3|V|-|L|}{2}$.
\end{proof}

An \emph{ordered graph} is a graph equipped with an order on its vertices and an order on the half-edges of each vertex. 
\begin{rem}
\label{remark: ordered graphs have no automorphisms}
Graphs, and even ordered graphs, can have automorphisms, but a leaf-free ordered graph has only the identity automorphism.
\end{rem}

\subsection{The core of a trivalent graph}
\begin{dfn}
Let $\Gamma$ be a trivalent graph. 
A half-edge $h$ of $\Gamma$ is \emph{superfluous} if either 
\begin{enumerate} 
\item the half-edge $h$ is a leaf or
\item the half-edge $h$ ``points toward a tree'' in the sense that there is a simply connected component of $\Gamma\setminus\{h,\iota(h)\}$ which does not contain $\tau(h)$ but is in the same connected component as $\tau(h)$ in the larger ambient graph $\Gamma$.
\end{enumerate}
A vertex is superfluous if it has a superfluous half-edge and is otherwise \emph{core}.
\end{dfn}
\begin{dfn}
Let $\Gamma$ be a (possibly ordered) trivalent graph. The \emph{core} of $\Gamma$ is the graph $\Gamma'=(V',H',\tau',\iota')$ (ordered if $\Gamma$ is) whose
\begin{enumerate}
\item vertices are the core vertices of $\Gamma$, 
\item half-edges are the half-edges of the core vertices, 
\item whose target map $\tau'$ is induced,
\item whose half-edge bijection $\iota'$ is induced by the bijection $\iota$ of $\Gamma$ in the sense explained below, and
\item whose orders on vertices and half-edges are induced by those of $\Gamma$.
\end{enumerate}
For $h$ a half-edge of a core vertex $v_0$, construct a sequence 
\[(h=h_0,h_1,\ldots, h_n=\iota'(h))\] 
of half-edges recursively as follows.
Suppose $(h_0,\ldots, h_{2k})$ is constructed.
Let $h_{2k+1}=\iota(h_{2k})$. 

If $\tau(h_{2k+1})$ is a core vertex, then $n=2k+1$ and we are done.

Otherwise, $\tau(h_{2k+1})$ is superfluous and thus has at least one superfluous half-edge. 
The half-edge $h_{2k+1}$ cannot be superfluous because then $v_0$ would be superfluous (it would imply $v_0$ is a vertex of a subtree in $\Gamma\setminus \{h_{2k},h_{2k+1}\}$, which in turn would imply that both half-edges of $v_0$ other than $h_{0}$ were superfluous).
On the other hand, if both the remaining half-edges of $\tau(h_{2k+1})$ were superfluous, then $h_{2k}$ would be superfluous as well.
Then since $\tau(h_{2k+1})$ is superfluous, $h_{2k+1}$ is not superfluous, and at least one of the other two half-edges of $\tau(h_{2k+1})$ is not superflous, we conclude that $\tau(h_{2k+1})$ has precisely one superfluous half-edge.
Let $h_{2k+2}$ be the (necessarily unique) non-superfluous half-edge of $v$ different from $h_{2k+1}$. 
\end{dfn}
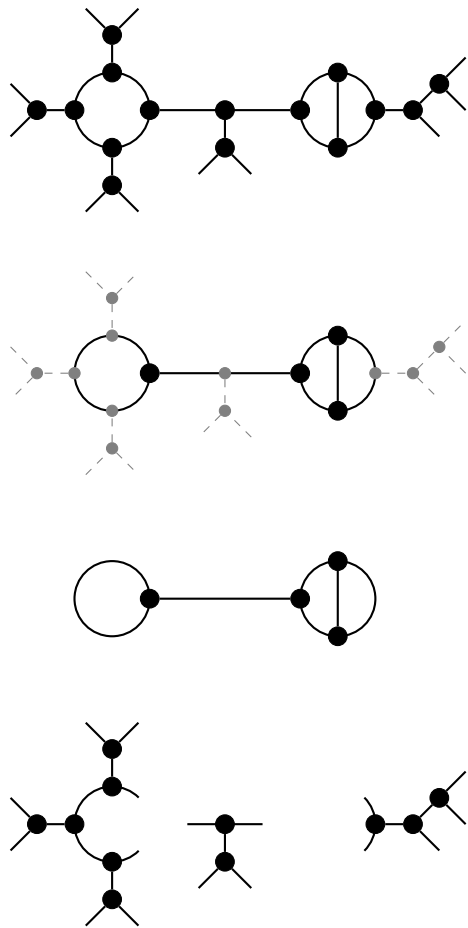
\begin{figure}\label{figure: core}
\begin{tikzpicture}
\node[circle,fill=black,inner sep=2.5pt,draw] (a) at (-1,0) {};
\node[circle,fill=black,inner sep=2.5pt,draw] (b) at (1,0) {};
\node[circle,fill=black,inner sep=2.5pt,draw] (c) at (2,0) {};
\node[circle,fill=black,inner sep=2.5pt,draw] (d) at (1.5,-.5) {};
\node[circle,fill=black,inner sep=2.5pt,draw] (e) at (1.5,.5) {};
\node[circle,fill=black,inner sep=2.5pt,draw] (g) at (-1.5,-.5) {};
\node[circle,fill=black,inner sep=2.5pt,draw] (h) at (-1.5,.5) {};
\node[circle,fill=black,inner sep=2.5pt,draw] (i) at (-2,0) {};
\node[circle,fill=black,inner sep=2.5pt,draw] (j) at (-1.5,-1) {};
\node[circle,fill=black,inner sep=2.5pt,draw] (k) at (-1.5,1) {};
\node[circle,fill=black,inner sep=2.5pt,draw] (l) at (-2.5,0) {};
\node[circle,fill=black,inner sep=2.5pt,draw] (m) at (2.5,0) {};
\node[circle,fill=black,inner sep=2.5pt,draw] (n) at (0,-.5) {};
\node[circle,fill=black,inner sep=2.5pt,draw] (o) at (0,0) {};
\node[circle,fill=black,inner sep=2.5pt,draw] (p) at (2.85,.35) {};

\draw[thick] (1.5,0) circle (.5cm);
\draw[thick] (-1.5,0) circle (.5cm);
\draw[thick] (a) edge (b);
\draw[thick] (o) edge (n);
\draw[thick] (i) edge (l);
\draw[thick] (j) edge (g);
\draw[thick] (k) edge (h);
\draw[thick] (m) edge (p);
\draw[thick] (c) edge (m);
\draw[thick] (e) edge (d);
\draw[thick] (m)--(2.85,-.35);
\draw[thick] (3.2,0)--(p)--(3.2,.7);
\draw[thick] (-2.85,.35)--(l)--(-2.85,-.35);
\draw[thick] (.35,-.85)--(n)--(-.35,-.85);
\draw[thick] (-1.85,-1.35)--(j)--(-1.15,-1.35);
\draw[thick] (-1.85,1.35)--(k)--(-1.15,1.35);
\begin{scope}[yshift = -3.5cm]
\node[circle,fill=black,inner sep=2.5pt,draw] (a) at (1,0) {};
\node[circle,fill=black,inner sep=2.5pt,draw] (b) at (-1,0) {};
\draw[thick] (1.5,0) circle (.5cm);
\draw[thick] (-1.5,0) circle (.5cm);
\draw[thick] (a) edge (b);
\node[circle,fill=gray,gray,inner sep=1.5pt,draw] (c) at (2,0) {};
\node[circle,fill=black,inner sep=2.5pt,draw] (d) at (1.5,-.5) {};
\node[circle,fill=black,inner sep=2.5pt,draw] (e) at (1.5,.5) {};
\node[circle,fill=gray,gray,inner sep=1.5pt,draw] (g) at (-1.5,-.5) {};
\node[circle,fill=gray,gray,inner sep=1.5pt,draw] (h) at (-1.5,.5) {};
\node[circle,fill=gray,gray,inner sep=1.5pt,draw] (i) at (-2,0) {};
\node[circle,fill=gray,gray,inner sep=1.5pt,draw] (j) at (-1.5,-1) {};
\node[circle,fill=gray,gray,inner sep=1.5pt,draw] (k) at (-1.5,1) {};
\node[circle,fill=gray,gray,inner sep=1.5pt,draw] (l) at (-2.5,0) {};
\node[circle,fill=gray,gray,inner sep=1.5pt,draw] (m) at (2.5,0) {};
\node[circle,fill=gray,gray,inner sep=1.5pt,draw] (n) at (0,-.5) {};
\node[circle,fill=gray,gray,inner sep=1.5pt,draw] (o) at (0,0) {};
\node[circle,fill=gray,gray,inner sep=1.5pt,draw] (p) at (2.85,.35) {};
\draw[gray,dashed] (o) edge (n);
\draw[gray,dashed] (i) edge (l);
\draw[gray,dashed] (j) edge (g);
\draw[gray,dashed] (k) edge (h);
\draw[gray,dashed] (m) edge (p);
\draw[gray,dashed] (c) edge (m);
\draw[thick] (e) edge (d);
\draw[gray,dashed] (m)--(2.85,-.35);
\draw[gray,dashed] (3.2,0)--(p)--(3.2,.7);
\draw[gray,dashed] (-2.85,.35)--(l)--(-2.85,-.35);
\draw[gray,dashed] (.35,-.85)--(n)--(-.35,-.85);
\draw[gray,dashed] (-1.85,-1.35)--(j)--(-1.15,-1.35);
\draw[gray,dashed] (-1.85,1.35)--(k)--(-1.15,1.35);
\end{scope}
\begin{scope}[yshift = -6.5cm]
\node[circle,fill=black,inner sep=2.5pt,draw] (a) at (1,0) {};
\node[circle,fill=black,inner sep=2.5pt,draw] (b) at (-1,0) {};
\draw[thick] (1.5,0) circle (.5cm);
\draw[thick] (-1.5,0) circle (.5cm);
\draw[thick] (a) edge (b);
\node[circle,fill=black,inner sep=2.5pt,draw] (d) at (1.5,-.5) {};
\node[circle,fill=black,inner sep=2.5pt,draw] (e) at (1.5,.5) {};
\draw[thick] (e) edge (d);
\end{scope}

\begin{scope}[yshift = -9.5cm]
\draw[thick] (-2,0) arc (180:45:.5);
\draw[thick] (-2,0) arc (180:315:.5);
\draw[thick] (2,0) arc (0:45:.5);
\draw[thick] (2,0) arc (0:-45:.5);
\draw[thick] (-.5,0)--(.5,0);

\node[circle,fill=black,inner sep=2.5pt,draw] (c) at (2,0) {};
\node[circle,fill=black,inner sep=2.5pt,draw] (g) at (-1.5,-.5) {};
\node[circle,fill=black,inner sep=2.5pt,draw] (h) at (-1.5,.5) {};
\node[circle,fill=black,inner sep=2.5pt,draw] (i) at (-2,0) {};
\node[circle,fill=black,inner sep=2.5pt,draw] (j) at (-1.5,-1) {};
\node[circle,fill=black,inner sep=2.5pt,draw] (k) at (-1.5,1) {};
\node[circle,fill=black,inner sep=2.5pt,draw] (l) at (-2.5,0) {};
\node[circle,fill=black,inner sep=2.5pt,draw] (m) at (2.5,0) {};
\node[circle,fill=black,inner sep=2.5pt,draw] (n) at (0,-.5) {};
\node[circle,fill=black,inner sep=2.5pt,draw] (o) at (0,0) {};
\node[circle,fill=black,inner sep=2.5pt,draw] (p) at (2.85,.35) {};
\draw[thick] (o) edge (n);
\draw[thick] (i) edge (l);
\draw[thick] (j) edge (g);
\draw[thick] (k) edge (h);
\draw[thick] (m) edge (p);
\draw[thick] (c) edge (m);
\draw[thick] (m)--(2.85,-.35);
\draw[thick] (3.2,0)--(p)--(3.2,.7);
\draw[thick] (-2.85,.35)--(l)--(-2.85,-.35);
\draw[thick] (.35,-.85)--(n)--(-.35,-.85);
\draw[thick] (-1.85,-1.35)--(j)--(-1.15,-1.35);
\draw[thick] (-1.85,1.35)--(k)--(-1.15,1.35);
\end{scope}
\end{tikzpicture}
\caption{A graph $\Gamma$, the same graph $\Gamma$ with superfluous vertices and half-edges in thinner gray, the core of $\Gamma$, and the insertion forest of $\Gamma$.}
\end{figure}

\begin{lem}
The core is a well-defined leaf-free trivalent graph.
\end{lem}
\begin{proof}
The procedure to generate $h_n$ cannot repeat a half-edge and so terminates.
If $(h_0,\ldots, h_n)$ is the involution sequence for $h_0$ then $(h_n,\ldots, h_0)$ is the sequence for $h_n$, so $\iota'$ is an involution.
Because of this reflection property, for $h_n$ to equal $h_0$ would either require some $h_k$ to be a leaf or $h_{2k+2}=h_{2k+1}$ for some index, both of which are false by construction.
Therefore $\iota'$ is fixed-point free.
\end{proof}
\begin{lem}
Let $\Gamma$ be a connected trivalent graph with Euler characteristic $\chi$. 
Then the core of $\Gamma$ is empty if $\chi\ge 0$ and has Euler characteristic $\chi$ otherwise.
A graph with $m$ leaves such that every connected component has non-positive Euler characteristic must have $m$ superfluous vertices.
\end{lem}
\begin{proof}
If $\Gamma$ is a tree then every vertex is trivially superfluous.
If $\Gamma$ has Euler characteristic $0$ then there is a single cycle in $|\Gamma|$ and so every vertex must have a half-edge pointing ``away'' from the cycle. So every vertex is superfluous.

The realization of any trivalent graph of negative Euler characteristic contains either an embedded ``theta'' or an embedded ``handcuff''. 
In either case the trivalent vertices of these subspaces must be core. 
Therefore a graph with negative Euler characteristic contains a core vertex.

Now divide the superfluous vertices into two types, according to whether the vertex has half-edges appearing in a sequence to generate the involution $\iota'$ or not. 
If $v$ is of the first type, then it has precisely two half-edges $h_1$ and $h_2$ in such a sequence. 
The third half-edge $h_3$ is necessarily superfluous, and as a result is attached to a tree (whose vertices are necessarily of the second type). 
Since there are core vertices, this must exhaust all superfluous vertices by connectedness.

Then in this case, the subgraph spanned by superfluous vertices is a forest subgraph of $\Gamma$, and the procedure to generate the core replaces each tree in the union with a ``half-edge gluing''.
Removing a subtree from a graph increases the Euler characteristic by one. 
Identifying two leaves into an edge decreases the Euler characteristic by one.
Therefore these two operations together preserve the Euler characteristic. 
\end{proof}
\begin{dfn}
Let $v$ be a superfluous vertex of a trivalent graph $\Gamma$ which is in a connected component with nonempty core.
The \emph{insertion forest} of $v$ is a graph built as follows.
The vertices are of the superfluous vertices of $\Gamma$. 
The half-edges are the half-edges of the superfluous vertices of $\Gamma$.
The target map is induced.
For the involution $\iota_{if}$, let $h$ be a half-edge of the insertion forest. 
If $\iota(h)$ is also in the insertion forest, then $\iota_{if}(h)=\iota(h)$.
Otherwise we declare that $\iota_{if}(h)=h$, i.e., that $h$ is a leaf of the insertion forest.
We call a connected component of the insertion forest an \emph{insertion tree}.
\end{dfn}
\begin{lem}
Let $T$ be an insertion tree of a trivalent graph $\Gamma$.
Then $T$ has precisely two leaves which are not leaves of $\Gamma$.
\end{lem}
\begin{proof}
These are the leaves that are paired under $\iota$ with half-edges in the core. 
If there were none this would violate connectedness.
If there were only one, $h$, then $\iota(h)$ would be superfluous so $\tau(\iota(h))$ would be superfluous, a contradiction.
If there were more than two, then superfluity would fail for some vertex of $T$.
\end{proof}

\section{Generating functions}
\label{app:series} 
This appendix is concerned with calculating generating functions for classes of graphs involved in the master series $\masterseries$ and the key lemma~\ref{lem: evaluation of master series}. 

We use $\mathfrak{P}(n)$ to denote the set of unordered partitions of the finite set $\{1,\ldots, n\}$. That is, an element of $\mathfrak{P}(n)$ is a set $\mathfrak{p}$ of pairwise disjoint nonempty subsets (called \emph{blocks}) $\{\mathfrak{b}_i\}$ of $\{1,\ldots,n\}$ whose union is $\{1,\ldots,n\}$.

We start with a counting lemma and a standard combinatorial technique that we will use repeatedly.
\begin{lem}
\label{lem: count of rooted ordered trivalent trees}
The number of isomorphism classes of rooted ordered trivalent trees with $n$ vertices (for $n\ge 1$) is 
\[
\frac{(2n)!}{(n+1)!}3^n.
\]
\end{lem}
\begin{proof}
It is a standard combinatorial identity that the number of planar binary trees with $n$ vertices is $\frac{1}{n+1}\binom{2n}{n}$ for $n\ge 1$.
Since half-edges at each vertex are ordered, we have a choice of which half-edge points toward the root in each of the $n$ vertices, which yields a factor of $3^n$.
The planar structure of the tree induces an order on the other two half-edges of each vertex, and so specifying that this order and the given order must be compatible with respect to some convention uses the planarity to keep track of the distinction between the other two half-edges.
The vertices must be ordered which gives another factor of $n!$ and multiplying all of these together yields the result.
\end{proof}
Sometimes it is useful to include a ``trivial tree'' for the case $n=0$ and sometimes it is not. 
Our default will follow our conventions and we will specify explicitly if we want to include a trivial tree in this count. 

The following ``compositional formula'' will be useful several times.
\begin{thm}[{Compositional formula~\cite[5.1.4]{EN}}]
\label{thm: compositional formula}
Given two exponential generating functions
\begin{align*}
F(z)&= \sum_{n=1}^\infty \frac{f_n}{n!}z^n\\
G(z)&= 1+\sum_{n=1}^\infty \frac{g_n}{n!}z^n,
\end{align*}
then the exponential generating function for the sequence
\begin{align*}
h_0&=1
\\
h_n&=\sum_{\mathfrak{p}\in \mathfrak{P}(n)}g_{|\mathfrak{p}|}\prod_{\mathfrak{b}_i\in \mathfrak{p}}f_{|\mathfrak{b}_i|}
\end{align*}
is $(G\circ F)(z)$.
\end{thm}
One well-known special case of this theorem occurs when $g_n=1$ for all $n$. 
\begin{cor}[Exponential compositional formula]
\label{cor: exponential compositional formula}
Let $F(z)$ be the exponential generating function for $f_n$ for $n\ge 1$ as above.
Then $e^{F(z)}$ is the exponential generating function for
\begin{align*}
h_0&=1
\\
h_n&=\sum_{\mathfrak{p}\in \mathfrak{P}(n)}\prod_{\mathfrak{b}_i\in \mathfrak{p}}f_{|\mathfrak{b}_i|}.
\end{align*}
\end{cor}
We will use this formulation to count possibly disconnected graphs of some sort in terms of a count of connected graphs of the same sort.

Now we proceed to the various counting of the graphs involved. 

\subsection*{Leaf-free, negative Euler characteristic}
In this section, we prove the following count.
\begin{lem}
\label{lem: neg euler gen series leaf free}
The exponential generating function $G_-^{lf}(x)$ for isomorphism classes of ordered leaf-free trivalent graphs with $n$ vertices such that each connected component has negative Euler characteristic is
\begin{equation}
\label{eq: neg euler gen series leaf free}
G_-^{lf}(x)=\sum_{n=0}^\infty \frac{(6n-1)!!}{(2n)!}x^{2n}.
\end{equation}
\end{lem}
\begin{proof}
By Lemma~\ref{lem: euler trivalent}, asking for negative Euler characteristic is no condition on the graph, so this is counting perfect pairings of $3n$ half-edges. 
There are no perfect pairings if $n$ is odd and $(3n-1)!!$ if $n$ is even. 
The denominator comes from making this an exponential generating function.
\end{proof}
\subsection*{Two-rooted trees}
In this section, we prove the following count.
\begin{lem}
\label{lem: insertion trees gen series}
The two variable generating function $T_{rr}(x,y)$ in Definition \ref{trees} for isomorphism classes of ordered trivalent trees with $n$ vertices, two distinct distinguishable leaves called roots, and $m$ non-root leaves is
\begin{equation}
\label{eq: insertion trees gen series}
T_{rr}(x,y)=\frac{1}{\sqrt{1-12xy}}-1.
\end{equation}
\end{lem}
\begin{proof}
We use Lemma~\ref{lem: count of rooted ordered trivalent trees}.
A tree with $n$ vertices has $n+2$ leaves including the first root. 
Choosing a second root amounts to $n+1$ choices. 
Then the number of vertices is the same as the number of non-root leaves. 
We want an exponential generating function, which introduces a factor of $\frac{1}{n!}$. 
We arrive at the following series, as promised:
\[
  \sum_{n=1}^\infty
\frac{(n+1)}{n!}\frac{(2n)!}{(n+1)!}(3xy)^n=
\sum_{n=1}^\infty \binom{2n}{n}(3xy)^n.\qedhere
  \]  
\end{proof}

\subsection*{Zero Euler characteristic}
Next, we prove the following count.
\begin{lem}
\label{lem: 0 euler gen series}
The generating function $G_0(x,y)$ for the isomorphism classes of ordered trivalent graphs with $n$ and $m$ leaves such that each connected component has Euler characteristic zero is
\begin{equation}
\label{eq: 0 euler gen series}
G_0(x,y)=(1-12xy)^{-\frac{1}{4}}
\end{equation}
\end{lem}
\begin{rem}
Having zero Euler characteristic means that the number of edges, $\frac{1}{2}(3n-m)$, must be equal to the number of vertices, $n$, which implies that the two indices $n$ and $m$ must be equal for the number to be non-zero. 
Then $G_0(x,y)$ can be written in terms of the single variable $z=xy$ as an exponential generating function.
\end{rem}
\begin{proof}
By Corollary~\ref{cor: exponential compositional formula}, it suffices to show that
\[
G_0^c(z)=-\frac{1}{4}\log(1-12z),
\]
where $G_0^c(z)$ is the exponential generating function for connected ordered trivalent graphs with $n$ vertices and zero Euler characteristic.

A graph of Euler characteristic zero has a unique cycle, so we can decompose the count of such graphs by the length of this cycle. 
Then each vertex in the cycle is part of a unique maximal subtree containing no other vertex of the cycle, and this exhausts the vertices of the graph. 
This implies that the count of such graphs with $n$ vertices is 
\begin{equation}
\label{eq counting euler char zero}
\sum_{\mathfrak{p}\in \mathfrak{P}(n)}\left(\prod_{\mathfrak{b}_i\in \mathfrak{p}}\#(\mathcal{T}_{Y}(|\mathfrak{b}_i|)) \right)\frac{|\mathfrak{p}-1|!}{2},
\end{equation}
where $\mathcal{T}_Y(n)$ counts ordered trivalent trees with $n$ vertices including a special vertex and an ordered pair of two distinct leaves of the special vertex.
This is the vertex in the unique cycle and the labeling corresponds to deciding ``which half-edge goes which way'' in that unique cycle. 
The $\frac{|\mathfrak{p}-1|!}{2}$ corresponds to the dihedral symmetry involved in assembling these trees.

We will use Theorem~\ref{thm: compositional formula}.
To use the formula directly requires one of the exponential generating functions to have constant term $1$ so we will artificially modify the series~\eqref{eq counting euler char zero} by putting a dummy $1$ in when $n=0$, instead of a $0$. 
Then $G_0^c(z)+1$ is the composition of the exponential generating function 
\[
1+\sum_{n=0}^\infty \frac{(n-1)!}{2n!}z^n
=
1-\frac{1}{2}\log(1-z)
\] 
with the exponential generating function
\[T_Y(z)=\sum \#\mathcal{T}_Y(n)\frac{z^n}{n!}\] 
for $\mathcal{T}_Y$.
Therefore to conclude we should compute $T_Y(z)$.

Counting $\mathcal{T}_Y(n)$ is the same as choosing a vertex and the ordered pair of half-edges---this is a factor of $6n$---and then assembling the rest of the $n-1$ vertices into a rooted ordered tree (including the possibility of the trivial tree if $n=1$). 
Then by Lemma~\ref{lem: count of rooted ordered trivalent trees}, there are $\frac{(2n-2)!}{n!}3^{n-1}$ such trees (here this is valid for $n\ge 1$), and combining with the $6n$ we get the exponential generating function
which has exponential generating function
\[
T_Y(z)=2\sum_{n=1}^\infty \frac{(2n-2)!}{(n-1)!n!}(3z)^n.
\]
Taking the derivative of this formal series yields $\frac{6}{\sqrt{1-12z}}$, so the exponential generating function is $-\sqrt{1-12z}+c$ for a constant of integration $c$ which is $1$ because there is no constant term in the exponential generating function $T_Y(z)$.

Then the composition formula of Theorem~\ref{thm: compositional formula} yields
\[
G_0^c(z)+1 = 1-\frac{1}{2}\log(1-(1-\sqrt{1-12z}))=1-\frac{1}{2}\log(\sqrt{1-12z})
\]
as desired.
\end{proof}


\subsection*{Positive Euler characteristic}
In this section, we prove the following count.
\begin{lem}
\label{lem: pos euler gen series}
The two variable generating function $G^c_+(x,y)$ for isomorphism classes of ordered trivalent trees with $n$ vertices is
\begin{equation}
\label{eq: pos euler gen series}
 G_+^c(x,y)=\frac{(1-\sqrt{1-12xy})^3(1+3\sqrt{1-12xy})}{864x^2}.
\end{equation}
\end{lem}
\begin{proof}
Again we use Lemma~\ref{lem: count of rooted ordered trivalent trees}. 
Since ordered trees have only the identity automorphism, this means that we can count the \emph{unrooted} ordered trivalent trees by dividing by the $n+2$ possible choices of root, so that 
\[
G_+^c(x,y)=y^2\sum_{n=1}^\infty \frac{(2n)!}{(n+2)!}(3xy)^n.
\]
It is standard that
\[
\sum_{n=0}^\infty \frac{(2n)!}{n!}(3z)^n = \frac{1}{\sqrt{1-12z}}
\]
and this is $1$ less than the second derivative of $x^2G_+^c(x,y)$ viewed as a series in the single variable $z=xy$.
Integrating twice with respect to $z$, we find that 
\[
x^2G_+^c(x,y)= \frac{(1-12xy)^{\frac{3}{2}}}{108}-\frac{(xy)^2}{2} + \frac{xy}{6}-\frac{1}{108},
\]
where the constants of integration have been chosen to give the correct overall values (i.e., zero) for $(xy)^0$ and $(xy)^1$ on the right side.
Rewriting in terms of $\sqrt{1-12xy}$ (which will be more convenient later) yields the result.
\end{proof}

\section{Proof of the key lemma}
\label{appendix: key lemma proof}
In this section we prove Lemma~\ref{lem: evaluation of master series}, which is about a certain ``evaluation'' of the master series $\masterseries$.

For this purpose, we will use the calculations of the constituent generating functions of the master series $\masterseries$ from Appendix~\ref{app:series} (i.e., Lemmas~\ref{lem: neg euler gen series leaf free},~\ref{lem: insertion trees gen series},~\ref{lem: 0 euler gen series}, and~\ref{lem: pos euler gen series}).
These computations in hand, we perform the evaluation of Lemma~\ref{lem: evaluation of master series} variable by variable, starting with $z$ and $u$.

\begin{notation}
We will use $\sqrt{1-12xy}$ often, so we use $Q$ as shorthand for it.
\end{notation}
\begin{lem}
\label{lem: ev zu}
For any $n\ge 0$, the ``evaluation'' of the series in $x$, $y$, $z$, and $u$
\[
\frac{e^{zT_{r,r}(x,y)}}{G^{\lf}_-(xu)}
\]
obtained by linearly replacing $z^{n_3}u^{n_5}$ with $n_3!\binom{\frac{3}{2}n_5+3n}{n_3}$
is the following series in $x$ and $y$:
\[
\frac{Q^{-3n}}{G^{\lf}_-(xQ^{-\frac{3}{2}})}.
\]
\end{lem}
\begin{proof}
First, by Lemma~\ref{lem: insertion trees gen series},
\begin{align*}
e^{zT_{r,r}(x,y)}&=\sum_{j=0}^\infty \frac{z^j}{j!}\left(\frac{1}{Q}-1\right)^j
\end{align*}
so substituting $j!\binom{\frac{3}{2}n_5+3n}{j}$ for $u^{n_5}z^j$ yields
\[
\sum_{j=0}^\infty \binom{\frac{3}{2}n_5+3n}{j}\left(\frac{1}{Q}-1\right)^j
= \left(\frac{1}{Q}\right)^{\frac{3}{2}n_5+3n}.\qedhere
\]
\end{proof}
Next we turn to $w$, which is easy.
\begin{lem}
\label{lem: ev w}
The evaluation of the series $\frac{1}{e^{wG_0^c(x,y)}}$ in $w$, $x$, and $y$ at $w=3v+6n-3$ is $Q^{\frac{3v-3}{2} + 3n}$.
\end{lem}
\begin{proof}
By Lemma~\ref{lem: 0 euler gen series}, $G_0(x,y)=Q^{-\frac{1}{2}}$, so
\[
e^{-w\log(G_0(x,y))}=G_0(x,y)^{-w} = Q^{\frac{w}{2}}.\qedhere
\]
\end{proof}
We can combine these to achieve the following:
\begin{cor}
\label{cor: simplified evaluation}
Define  a series in $\simplifiedmasterseriestwo$ in $x$ by starting from the following series in $x$, $y$ and $v$
\[
\frac{Q^{\frac{3v-3}{2}}}{G_+(x,y)}
\]
and performing the following linear substitution:
\[
x^{n_1}y^{n_2}v^{n_3}\mapsto
\begin{cases}
0 & n_2\text{ odd,}\\
(n_2-1)!! x^{n_1}(n_1)^{n_3} & n_2\text{ even.} 
\end{cases}
\]
Then the simplified master series $\simplifiedmasterseries$ of Lemma~\ref{lem: evaluation of master series} is related to $\simplifiedmasterseriestwo$ via
\[
\simplifiedmasterseries(x)=\frac{\simplifiedmasterseriestwo(x)}{G_-^{\lf}(x)}.
\]
\end{cor}
\begin{proof}
By Lemmas~\ref{lem: ev zu} and~\ref{lem: ev w}, the simplified master series $\simplifiedmasterseries$ is obtained from 
\[
\frac{Q^{\frac{3v-3}{2}}}{G_+(x,y)G_-^{\lf}(xQ^{-\frac{3}{2}})}
\]
by performing the indicated substitution.
In order to compare terms, write 
\begin{align*}
\frac{1}{G^{\lf}_-(x)}&=\sum_n \gamma_n x^{n}.
\end{align*}
Then the quotient can be written
\[
\sum_{n} \gamma_n \frac{Q^{\frac{3(v-n)-3}{2}}}{G_+(x,y)} x^{n}. 
\]
If we write $v'=v-n$, then we are replacing $v'$ in the $x^{n_1+n}$ term of the series expansion of the expression
\[
\gamma_n \frac{Q^{\frac{3v'-3}{2}}}{G_+(x,y)}x^{n}
\]
with $n_1$ (and also doing a substitution for $y$).
That's the same as replacing $v$ in the $x^{n_1}$ term of the series expansion of the expression
\[
\gamma_n \frac{Q^{\frac{3v-3}{2}}}{G_+(x,y)}
\]
with $n_1$, doing the substitution for $y$, and then multiplying by $x^n$ at the end. Then the $\gamma_n$ and $x^n$ terms have become decoupled and the overall effect is to do the substitution on the simpler quotient $\frac{Q^{\frac{3v-3}{2}}}{G_+(x,y)}$ and then divide by $G_-^{\lf}(x)$.
\end{proof}
Now the key lemma can be derived from the following technical result.
\begin{lem}
\label{lem: diagonal coefficient}
The coefficient of $(xy)^{2n}$ in the expansion of 
\[Q^{\frac{6n-3}{2}}(1+Q)^{2n+1}(2Q+1)^{-\frac{2n+1}{2}}\] 
is
\begin{equation*}
\frac{(6n)!n!}{(3n)!(2n)!(2n)!}\frac{2}{\sqrt{3}}\left(\frac{1}{3}\right)^n.
\end{equation*}
\end{lem}
Assuming this lemma, we finish the main proof.
\begin{proof}[Proof of Lemma~\ref{lem: evaluation of master series}]
By Corollary~\ref{cor: simplified evaluation}, it suffices to perform the substitution of that corollary on the series expansion of the expression
\[
\frac{Q^{\frac{3v-3}{2}}}{G_+(x,y)}
\]
and verify that the result is equal to $G_-^{\lf}(x)$.
Lemma~\ref{lem: pos euler gen series} says 
\[
G_+(x,y)=\exp\left(\frac{(1-Q)^3(1+3Q)}{864x^2}\right).
\]
So the parity of $x$ and $y$ in the series in question will always be the same. 
Then after substituting for $y$, there will be no odd powers of $x$, so we may confine ourselves to the even powers.
In fact, the power of $y$ and the power of $x$ differ only in terms of the $x^2$ in the denominator of $G^c_+(x,y)$ so if we are interested in the $2n_1$ power of $x$ in the substitution, it is the $x^{2n_1}$ term of the series
\[
Q_x^{\left(\frac{6n_1-3}{2}\right)}\sum_{n=0}^\infty \frac{1}{n!}(2n_1+2n-1)!!\left(-\frac{(1-Q_x)(1+3Q_x)}{864x^2}\right)^n,
\]
where $Q_x=\sqrt{1-12x}$. 
But
\[
\sum_{n=0}^\infty \frac{(2n_1+2n-1)!!}{n!}x^n=\frac{(2n_1-1)!!}{(1-2x)^{\frac{2n_1+1}{2}}}
\]
so the series of interest is
\[
Q_x^{\left(\frac{6n_1-3}{2}\right)}\frac{(2n_1-1)!!}{\left(1+\frac{(1-Q_x)(1+3Q_x)}{432x^2}\right)^{\frac{2n_1+1}{2}}}.
\]
By direct manipulation we have the equation
\[
1+
\frac{(1-Q_x)^3(1+3Q_x)}{432x^2}
=
\frac{4}{3}\frac{(1+2Q_x)}{(1+Q_x)^2}.
\]
Then finally we are interested in the $x^{2n_1}$ coefficient of
\[
(2n_1-1)!!\left(\frac{3}{4}\right)^{\frac{2n_1+1}{2}}\frac{Q_x^{\frac{6n_1-3}{2}}(1+Q_x)^{2n_1+1}}{(1+2Q_x)^{\frac{2n_1+1}{2}}}.
\]
By Lemma~\ref{lem: diagonal coefficient}, this coefficient is
\[
(2n_1-1)!!\left(\frac{3}{4}\right)^{\frac{2n_1+1}{2}}\frac{(6n_1)!n_1!}{(3n_1)!(2n_1)!(2n_1)!}\frac{2}{\sqrt{3}}\left(\frac{1}{3}\right)^{n_1}
=
\frac{(6n_1-1)!!}{(2n_1)!},
\]
which is by Lemma~\ref{lem: neg euler gen series leaf free} the $x^{2n_1}$ coefficient of $G^{\lf}_-(x)$, as desired.
\end{proof}
It only remains to prove Lemma~\ref{lem: diagonal coefficient}.

\begin{proof}[Proof of Lemma~\ref{lem: diagonal coefficient}]
Let us begin with an overview. The proof will use a method of Hautus and Klarner~\cite{HK} for extracting the ``diagonal'' of an analytic series in two variables. 
One expositional option would be to build such a two-variable series whose diagonal was precisely the generating function of the numbers we care about.
We have chosen instead to build a series whose \emph{even} diagonal entries are the numbers we care about and whose odd entries are irrelevant for our purposes. This allows us a little more flexibility in the shape of the series we build and we can apply a special case of Hautus and Klarner.
The price we pay for this is that after finishing the computation we have an extra step to extract the even degree coefficients.

Let us begin. 
It is convenient to write our series in terms of $U=(1-12x)^{-\frac{1}{2}}$ (i.e., the reciprocal of $Q_x$).
Consider the following series in $x$ and $t$:
\begin{align*}
\label{eq: two variable series}
F(x,t)&=\sum_{n=0}^\infty \frac{1}{U}^{\frac{3n-3}{2}}\left(1+\frac{1}{U}\right)^{n+1}\left(\frac{2}{U}+1\right)^{-\frac{n+1}{2}}t^n
\\&=U(U+1)(U+2)^{-\frac{1}{2}}\sum_{n=0}^\infty \left(\frac{U+1}{U^2\sqrt{U+2}}t\right)^n
\end{align*}
Note that the even ``diagonal coefficients'', i.e., the coefficients of $x^{2n}t^{2n}$ in this series, are indeed the quantities of interest to us.

The auxiliary functions 
\begin{align*}
W(x)&=\frac{U(U+1)}{\sqrt{U+2}},\\
Y(x)&=\frac{U+1}{U^2\sqrt{U+2}}=\frac{W(x)}{U^3}
\end{align*} 
are themselves analytic in a neighborhood of $x=0$. 
Then the special case of the result of Hautus and Klarner~\cite[Section 4]{HK} allows us to extract the one variable diagonal series associated to $F(x,t)$. 
That is, if 
\[
F(x,t)=\sum_{m=0}^\infty\sum_{n=0}^\infty \alpha_{m,n}x^m t^n,
\]
then we can extract
\[
F_\Delta(z)=\sum_{n=0}^\infty \alpha_{n,n}z^n
\]
(the variable $z$ here is unrelated to the variable $z$ used in the master series and other similar series earlier).
The recipe of Hautus and Klarner says that $F_\Delta(z)$ is given via a residue computation which simplifies in the case of interest to be
\[
F_\Delta(z)=\frac{W(x)}{1-zY'(x)}\Biggr|_{z=\frac{x}{Y(x)}},
\]
where $x$ is defined implicitly by $zY(x)=x$ as an analytic function of $z$ in a neighborhood of the origin.
It's convenient for us to rewrite this 
\[
F_\Delta(z)=\frac{W(x)/Y(x)}{\frac{dz}{dx}}\Biggr|_{z=\frac{x}{Y(x)}}.
\]
Now using $x=\frac{U^2-1}{12U^2}$, we see that the equality $zY(x)=x$ occurs at
\begin{equation}
\label{implicit z equation}
12z= \frac{U^2-1}{U^2}\frac{U^2\sqrt{U+2}}{U+1}=(U-1){\sqrt{U+2}}.
\end{equation}
and so in our case, we can use the facts that $W(x)=U^3 Y(x)$ and that $\frac{dU}{dx}=6U^3$ along with the chain rule to get
\[
F_\Delta(z)=\frac{U^3}{\frac{dz}{dU}\cdot 6U^3}\Biggr|_{z=\frac{x}{Y(x)}}=\frac{1}{6\frac{dz}{dU}}\Biggr|_{z=\frac{x}{Y(x)}},
\]
which is the formal value of $U'(z)/6$ where $U$ and $z$ are related by Equation~\eqref{implicit z equation}.
Now by trigonometric identities or by solving a cubic, we find that near $(z=0,U=1)$ we can give an explicit form for the inverse function:
\[
U(z)=2\cos\left(\frac{2}{3}\cos^{-1}(6z)\right),
\]
where we take the standard branch of the arccosine.

It is a tedious but elementary verification that both $U'(z)/6$ and the convergent series
\[
\sum_{n=0}^\infty \frac{(6n)!n!}{(3n)!(2n)!(2n)!}\left(\frac{z}{3}\right)^{2n}
\]
are solutions to the second order linear differential equation
\[
20 y + 108z y'+ (36z^2-1) y'' = 0.
\]
Because the coefficients of $y$ and $y''$ are even functions of $z$ and the coefficient of $y'$ is an odd function, $U'(-z)/6$ is also a solution, so the average $\frac{1}{12}(U'(z)+U'(-z))$ and the given series are both solutions with vanishing first derivative at $0$. Then they agree up to a scalar multiple, which can be checked by hand to be $\frac{2}{\sqrt{3}}$.
\end{proof}

\bibliographystyle{amsplain}
\bibliography{mybibliography}
\end{document}